\documentclass[11pt,a4paper,leqno]{amsart} 
\usepackage[all]{xy}

\usepackage[utf8]{inputenc}
\usepackage{amssymb}
\usepackage{amsmath}
\usepackage{amsthm}
\usepackage{varioref}
\usepackage{dsfont}
\usepackage{cases}
\usepackage{pict2e}
\usepackage{stmaryrd}
\usepackage{graphicx}
\usepackage{enumerate}

\RequirePackage{yhmath}

\DeclareMathOperator{\Span}{\textrm{Span}}

\usepackage[a4paper]{geometry}
\begin{document}

\title{Strongly $n$-supercyclic operators}
\author{ERNST Romuald}
\date{}

\maketitle

\theoremstyle{plain}
\newtheorem{theo}{Theorem }[section]
\newtheorem{lem}[theo]{Lemma}
\newtheorem{cor}[theo]{Corollary}
\newtheorem{prop}[theo]{Proposition}

\theoremstyle{definition}
\newtheorem{defi}[theo]{Definition}
\newtheorem{rem}[theo]{Remark}

\newtheorem{exe}[theo]{Example}
\newtheorem{cexe}[theo]{Counterexample}
\newtheorem*{nota}{Notation}
\newtheorem*{quest}{Question}

\newtheorem*{thmbfs}{Theorem Bourdon, Feldman, Shapiro}
\newtheorem*{thmf}{Theorem Feldman}
\newtheorem*{thmme}{Theorem \cite{Ernstnsupstrongnsupop}}
\newtheorem*{thmhc}{Theorem Hypercyclicity Criterion}
\newtheorem*{thmsc}{Theorem Supercyclicity Criterion}
\newtheorem*{thmct}{Theorem Circle Theorem}
\newtheorem*{thmas}{Theorem Ansari-Shkarin}

\newcommand{\claim}{\vspace{0.5cm}\noindent{\bf Claim:}\quad}
\newcommand{\R}{\mathbb R}
\newcommand{\Z}{\mathbb Z} 
\newcommand{\N}{\mathbb N}
\newcommand{\Q}{\mathbb Q}
\newcommand{\C}{\mathbb C}
\newcommand{\K}{\mathbb K}
\newcommand{\Pn}{\mathbb P}
\newcommand{\HC}{\mathcal{H}\mathcal{C}}
\newcommand{\SC}{\mathcal{S}\mathcal{C}}
\newcommand{\Oc}{\mathcal{O}}
\newcommand{\ES}{\mathcal{E}\mathcal{S}}
\newcommand{\epsi}{\varepsilon}

\begin{abstract}
In this paper, we are interested in the properties of a new class of operators, recently introduced by Shkarin, called strongly $n$-supercyclic operators. This notion is stronger than $n$-supercyclicity. We prove that such operators have interesting spectral properties and we give examples and counter-examples answering some questions asked by Shkarin.
\end{abstract}

\section{Introduction}

In what follows $X$ will denote completely separable Baire vector spaces over the field $\K=\R,\C$ and $T$ will be a bounded linear operator on $X$.
Since the last 1980's, density properties of orbits of operators have been of great interest for many mathematicians, particularly hypercyclic and cyclic operators for their link with the invariant subspace problem. Another reason explaining this interest is that they appear in many well-known classes of operators: weighted shifts, composition operators, translation operators,...
\begin{defi}
A vector $x\in\ X$ is said to be hypercyclic (resp. supercyclic) for $T$ if its orbit $\mathcal{O}(x,T):=\{T^{n}x,n\in\Z_{+}\}$ (resp. projective orbit $\{\lambda T^{n}x,n\in\Z_{+},\lambda\in\K\}$) is dense in $X$. The operator $T$ is said to be hypercyclic (resp. supercyclic) if it admits at least one hypercyclic (resp. supercyclic) vector.
\end{defi}
One may remove linearity in this definition, then under the same assumptions, $T$ is said to be universal. The definition of supercyclicity was introduced in 1974 by Hilden and Wallen \cite{Hilwal}. As one can see, this notion does not deal with orbits of vectors any more but with orbits of lines.
As we said before, these properties have been intensively studied and the reader may refer to \cite{Bay} and \cite{Grope} for a deep and complete survey. 
%
One of the main ingredient providing such operators is the so called Supercyclicity Criterion given by H.N. Salas \cite{Salas}, which is only a sufficient condition for supercyclicity.
\begin{thmsc}
Let $X$ be a separable Banach space and $T\in\mathcal{L}(X)$. $T$ satisfies the Supercyclicity Criterion if there exist a strictly increasing sequence $(n_{k})_{k\in\Z_{+}}$, two dense sets $\mathcal{D}_{1},\mathcal{D}_{2}\subset X$ in $X$ and a sequence of maps $S_{n_{k}}:\mathcal{D}_{2}\to X$ such that:
\begin{enumerate}[(a)]
\item $\Vert T^{n_{k}}x\Vert\Vert S_{n_{k}}y\Vert\to 0$ for any $x\in\mathcal{D}_{1}$ and $y\in\mathcal{D}_{2}$;\label{scrit1}
\item $T^{n_{k}}S_{n_{k}}y\to y$ for any $y\in\mathcal{D}_{2}$.\label{scrit2}
\end{enumerate}
If $T$ satisfies the Supercyclicity Criterion, then $T$ is supercyclic.
\end{thmsc}

This result is at the very heart of the theory. Indeed, only very few operators have been proved to be supercyclic without using this criterion.
Recently, some authors tried to generalise supercyclicity in a natural way. The first one is N. Feldman \cite{Felnsupop} at the beginning of the 2000's.
\begin{defi}
An operator $T$ is said to be $n$-supercyclic, $n\geq1$, if there is a subspace of dimension $n$ in $X$ with dense orbit.
\end{defi}
These operators have been studied in \cite{Bayhyp},\cite{Ber} and \cite{Bou} and \cite{Ernstnsupstrongnsupop} and many other articles.
 Feldman gave various classes of $n$-supercyclic operators and in particular the following:
\begin{exe}\label{exeFelnsupop}\cite{Felnsupop}
Let $n\in\mathbb{N}$. If $\{\Delta_{k},1\leq k\leq n\}$ is a collection of open disks,
$S_{k}=M_{z}$ on $L_{a}^{2}(\Delta_{k})$ for any $1\leq k\leq n$ and $S=\oplus_{k=1}^{n}S_{k}$, then $S^{\ast}$
is $n$-supercyclic.
\end{exe}
Then, in 2004, Bourdon, Feldman and Shapiro \cite{Bou} proved in the complex setting that $n$-supercyclicity is a purely infinite dimensional phenomenon.
\begin{theo}\label{theoboufelshap}
Let $n\geq2$. There is no $(n-1)$-supercyclic operator on $\C^{n}$. In particular, there is no $k$-supercyclic operator on $\C^{n}$ for any $1\leq k\leq n-1$.
\end{theo}
Recently, the present author \cite{Ernstnsupstrongnsupop} proved that things were different in the real setting. These theorems recall the behaviour of supercyclic operators in finite dimensional vector spaces.
Nevertheless, even though most of the supercyclic theorems have a $n$-supercyclic counterpart, some questions remain open. In particular, one may ask whether there exist a Birkhoff Theorem, a $n$-supercyclicity Criterion or even if the Ansari Theorem remains true for $n$-supercyclic operators. These questions are ``more difficult '' than the previous ones mainly because $X$ being a vector space, we are not considering a ``natural space`` for working on orbits of $n$-dimensional subspaces. In this spirit, in 2008, Shkarin \cite{Shkauniv} proposed the concept of strong $n$-supercyclicity requiring a stronger condition as its name suggests.
Let us first recall some well-known facts before coming to the definition of strong $n$-supercyclicity.
If $X$ has dimension greater than $n\in\N$, then one may define a topology on the $n$-th Grassmannian, denoted by $\Pn_{n}(X)$, which is the set of all $n$-dimensional subspaces of $X$. To do so, set $X_{n}$ the open set of all linearly independent $n$-tuples $x=(x_{1},\ldots,x_{n})\in X^{n}$ and endow $X_{n}$ with the topology induced by $X^{n}$.
Then set $\pi_{n}:X_{n}\rightarrow\Pn_{n}(X)$, $\pi_{n}(x)=\Span(x_{1},\ldots,x_{n})$ and define the topology on $\Pn_{n}(X)$ as being the coarsest for which $\pi_{n}$ is continuous and open. Now, let us move to the awaited definition:
\begin{defi}
Let $n\in\N$.
 An $n$-dimensional subspace of $X$ is said to be strongly $n$-supercyclic for $T$ if for every $k\in\Z_{+}$, $T^{k}(L)$ has dimension $n$ and if its orbit $$\mathcal{O}(L,T):=\{T^{n}(L),n\in\Z_{+}\}$$ is dense in $\Pn_{n}(X)$. The set of all strongly $n$-supercyclic subspaces for $T$ is denoted $\ES_{n}(T)$. The bounded linear operator $T$ is called strongly $n$-supercyclic if $\ES_{n}(T)\neq\emptyset$.
\end{defi}
\begin{rem}
 In this definition and all along this paper, we make no difference between $L$ as a subspace of $X$ and $L$ as an element of $\Pn_{n}(X)$.
\end{rem}
The main interest of this definition, compared to Feldman's one, is that it reduces to the universality of $T$ on the space $\Pn_{n}(X)$. From this point of view, the definition of strongly $n$-supercyclic operators seems quite natural.
Moreover, with this observation Shkarin \cite{Shkauniv} proved that strongly $n$-supercyclic operators do satisfy the Ansari property:
\begin{thmas}
Let $k,n\in\N$.
Then $\ES_{n}(T)=\ES_{n}(T^{k})$. In particular, $T$ is strongly $n$-supercyclic if and only if $T^{k}$ is strongly $n$-supercyclic.
\end{thmas}
When he introduced the previous definition, Shkarin asked the question whether $n$-supercyclicity is equivalent to strong $n$-supercyclicity. Indeed, a positive answer to this question solve the Ansari property problem for $n$-supercyclic operators. In fact, the present author gave a negative answer to this question in \cite{Ernstnsupstrongnsupop} and we will construct some more counterexamples in the present paper. 
Since, \cite{Shkauniv} is very concise on strong $n$-supercyclicity, giving only the definition and the Ansari property and \cite{Ernstnsupstrongnsupop} is only concerned with the finite dimensional setting, the aim of this paper is to present a deeper study of strong $n$-supercyclicity.

\section{Preliminaries and equivalent conditions to strong $n$-supercyclicity}

A useful Theorem in linear dynamics is Birkhoff's Transitivity Theorem because it permits to consider the ''orbit of an open set'' instead of the orbit of a point and is the key point to prove the Hypercyclicity and Supercyclicity Criteria. This property is called topological transitivity.
Such a result would be a stable anchor for studying strongly $n$-supercyclic operators and this is the purpose of this section.
But first, we are going to expose general properties that we need in the sequel and which allow one to express strong $n$-supercyclicity in a more concrete and handy way.
The following property is easy to check and allows one to work on the space $X^{n}$ instead of the space $X_{n}$ which is less structured.
\begin{prop}
$X_{n}$ is dense in $X^{n}$.
\end{prop}
%

\begin{rem}
 The following trivial fact is important in the sequel :
let $U$ be a non-empty open set in $X_{n}$ and $L$ be an $n$-dimensional subspace of $X$, then $(L\times\cdots\times L)\cap U\neq \emptyset$ $\Leftrightarrow$ $L\in\pi_{n}(U)$.
\end{rem}

Thanks to the link between $X_n$ and $X^n$, we are now able to characterise strong $n$-supercyclicity by density properties in $X^{n}$ rather than in $\Pn_n(X)$.
\begin{prop}\label{propbeq}
The following are equivalent:
\begin{enumerate}[(i)]
\item $T$ is strongly $n$-supercyclic;\label{beq1}
\item There exists a subspace $L$ of $X$ with dimension $n$ such that for every $i\in\Z_{+}$, $T^{i}(L)$ is $n$-dimensional and: \label{beq2}\\
 $\mathcal{B}:=\cup_{i=1}^{\infty}\pi_{n}^{-1}(T^{i}(L))$ is dense in $X^{n}$;
\item There exists a subspace $L$ of $X$ with dimension $n$ such that for every $i\in\Z_{+}$, $T^{i}(L)$ is $n$-dimensional and: \label{beq3}\\
 $\mathcal{E}:=\cup_{i=1}^{\infty}(T^{i}(L)\times\cdots\times T^{i}(L))$ is dense in $X^{n}$.
\end{enumerate}
\end{prop}

\begin{proof}
 We first prove that $(\ref{beq1})\Leftrightarrow(\ref{beq2})$ and then $(\ref{beq2})\Leftrightarrow(\ref{beq3})$

$(\ref{beq1})\Rightarrow(\ref{beq2}):$

Let $x=(x_{1},\ldots,x_{n})\in\ X_{n}$, $M:=\pi_{n}(x)\in\mathbb{P}_{n}(X)$ and $V$ be a non-empty open neighbourhood of $x$ in $X_{n}$. Since $\pi_{n}$ is open, then $W:=\pi_{n}(V)$ is an open neighbourhood of $M$ in $\mathbb{P}_{n}(X)$.
Moreover, strong $n$-supercyclicity of $T$ implies that there exists an $n$-dimensional subspace $L$ of $X$ such that: $\{T^{n}(L)\}_{n\in\mathbb{N}}$ is dense in $\mathbb{P}_{n}(X)$, thus there exists $k\in\mathbb{N}$ such that:
$T^{k}(L)\in W$.
Hence, there exists $y\in V$ such that 
$\pi_{n}(y)=T^{k}(L)$ and then $y\in\pi_{n}^{-1}(T^{k}(L))\subset\mathcal{B}$.
This proves the density of $\mathcal{B}$ in $X_{n}$ and thus in $X^n$.

$(\ref{beq1})\Leftarrow(\ref{beq2}):$

Assume that $\mathcal{B}$ is dense in $X^{n}$, the fact that $\mathcal{B}\subset X_{n}$ yields that $\mathcal{B}$ is dense in $X_{n}$.
Since $\pi_{n}$ is continuous and onto, $\pi_{n}(\mathcal{B})$ is dense in $\mathbb{P}_{n}(X)$.
Moreover:
\belowdisplayskip=0pt
$$
\pi_{n}(B)=\pi_{n}(\cup_{i=1}^{\infty}\pi_{n}^{-1}(T^{i}(L)))=\cup_{i=1}^{\infty}\pi_{n}(\pi_{n}^{-1}(T^{i}(L)))=\cup_{i=1}^{\infty}T^{i}(L)$$
Thus, we proved that $\cup_{i=1}^{\infty}T^{i}(L)$ is dense in $\mathbb{P}_{n}(X)$ and $T$ is strongly $n$-supercyclic.

$(\ref{beq2})\Rightarrow(\ref{beq3}):$

By definition of $\pi_{n}$, for any $k\in\mathbb{N}$, $\pi_{n}^{-1}(T^{k}(L))\subset T^{k}(L)\times\cdots\times T^{k}(L)\subset X^{n}$, thus $\mathcal{B}\subset\mathcal{E}$ and then $\mathcal{E}$ is dense in $X^{n}$.

$(\ref{beq2})\Leftarrow(\ref{beq3}):$

Let $U$ be a non-empty open set of $X^{n}$. Since $X_{n}$ is an open and dense subset of $X^{n}$, then the set $X_{n}\cap U$ is also non-empty and open in $X^{n}$ and since $\mathcal{E}$ is dense in $X^{n}$, there exists $x\in\mathcal{E}\cap X_{n}\cap U= X_{n}\cap (\cup_{i=1}^{\infty}T^{i}(L)\times\cdots\times T^{i}(L))\cap U$.
Hence there is $k\in\mathbb{N}$ such that $x\in X_{n}\cap (T^{k}(L)\times\cdots\times T^{k}(L))\cap U$, so $\pi_{n}(x)=T^{k}(L)$ and then $x\in\mathcal{B}\cap U$.
\end{proof}

\begin{rem}\label{remsommdir}
 In particular, (\ref{beq3}) above allows us to notice that if $T=T_1\oplus\cdots\oplus T_n$ on $X=E_1\oplus\cdots\oplus E_n$ is strongly $k$-supercyclic, then for any $i\in\{1,\ldots,n\}$, $T_i$ is strongly $k_{i}$-supercyclic where $k_{i}=\min(\dim(E_i),k)$.
\end{rem}

The last proposition makes possible to characterise the strongly $n$-supercyclic subspaces for an operator and shows that $\ES_n(T)$ is either empty or a G$_{\delta}$ subset of $\Pn_{n}(X)$.
Let us denote by $(V_{j})_{j\in\Z_{+}}$ an open basis of $X$.
\begin{prop}\label{propbespnsup}

$\mathcal{E}S_{n}(T)=\cap_{(j_{1},\ldots,j_{n})\in\mathbb{N}^{n}}\cup_{i\in\mathbb{N}}\pi_{n}((T\oplus\cdots\oplus T)^{-i}(V_{j_{1}}\times\cdots\times V_{j_{n}})\cap X_{n})$

\end{prop}

\begin{proof}
Let $L\in\mathcal{E}S_{n}(T)$, according to Proposition \ref{propbeq}, this means that $\cup_{i=1}^{\infty}T^{i}(L)\times\cdots\times T^{i}(L)$ is dense in $X^{n}$.
Then using the open basis this is equivalent to saying that:
$$\forall (j_{1},\cdots,j_{n})\in\mathbb{N}^{n},\exists i\in\mathbb{N} :\ (T^{i}(L)\times\cdots\times T^{i}(L))\cap (V_{j_{1}}\times\cdots\times V_{j_{n}}) \neq\emptyset.$$
Thus, $X_{n}$ being a dense open set of $X^n$, this can be re-written:
$$\forall (j_{1},\cdots,j_{n})\in\mathbb{N}^{n},\exists i\in\mathbb{N} :\ X_{n}\cap (L\times\cdots\times L)\cap (\oplus_{k=1}^{n} T)^{-i}(V_{j_{1}}\times\cdots\times V_{j_{n}})\neq\emptyset.$$ 
Finally, applying $\pi_{n}$ to the previous line gives the relation we expect:
$$L\in\ \cap_{(j_{1},\ldots,j_{n})\in\mathbb{N}^{n}}\cup_{i\in\mathbb{N}}\pi_{n}((T\oplus\cdots\oplus T)^{-i}(V_{j_{1}}\times\cdots\times V_{j_{n}})\cap X_{n}).$$
\end{proof}

At that point, it is possible to give a similar result as Birkhoff's Transitivity Theorem for the strong $n$-supercyclicity setting:
\begin{prop}\label{propbirkhoffnsup}
The following are equivalent:
\begin{enumerate}[(i)]
\item $T$ is strongly $n$-supercyclic;\label{birkhoffnsup4}
\item $\forall U\subset\mathbb{P}_{n}(X),\forall V\subset X^{n}$ open and non-empty, $\exists i\in\mathbb{N} :\ (\oplus_{k=1}^{n} T)^{i}(\pi_{n}^{-1}(U))\cap V\neq\emptyset$.\label{birkhoffnsup5}
\end{enumerate}
In particular, if $T$ is strongly $n$-supercyclic, then $\mathcal{E}S_{n}(T)$ is a G$_{\delta}$ dense subset of $\Pn_{n}(X)$.
\end{prop}

\begin{proof}
Let $L\in\mathcal{E}S_{n}(T)$. Since $X$ does not have any isolated point, $\Pn_{n}(X)$ does not have any either and then $\mathcal{O}(L,T)\subset\mathcal{E}S_{n}(T)$.
Thus, $\mathcal{E}S_{n}(T)$ is either empty or dense and is also a G$_{\delta}$ with Proposition \ref{propbespnsup}.
In particular, $T$ is strongly $n$-supercyclic if and only if $\mathcal{E}S_{n}(T)$ is dense in $\Pn_{n}(X)$, and using the characterisation of $\mathcal{E}S_{n}(T)$ from Proposition \ref{propbespnsup}, this means that for all non-empty open set $U\in\Pn_{n}(X)$ and any $(j_{1},\ldots,j_{n})\in\mathbb{N}^{n}$, there exists $i\in\mathbb{N}$ such that $$\pi_{n}((T\oplus\cdots\oplus T)^{-i}(V_{j_{1}}\times\cdots\times V_{j_{n}})\cap X_{n})\cap U\neq\emptyset$$ where $(V_{j})_{j\in\Z_{+}}$ is an open basis of $X$.\\
Thanks to the relation $\pi_{n}^{-1}(U)\cap X_{n}=\pi_{n}^{-1}(U)$, this can be re-written: for all non-empty open set $U\in\Pn_{n}(X)$, for any $(j_{1},\ldots,j_{n})\in\mathbb{N}^{n}$, there exists $i\in\mathbb{N}$ such that $$(T\oplus\cdots\oplus T)^{-i}(V_{j_{1}}\times\cdots\times V_{j_{n}})\cap \pi_{n}^{-1}(U)\neq\emptyset$$ and the proposition is proved.
\end{proof}

Thanks to these results, we are now able to prove the existence of strongly $n$-supercyclic operators. This is a first class of examples:
\begin{cor}\label{corpremex}
 Suppose that $T$ satisfies the Supercyclicity Criterion, then $T$ is strongly $n$-supercyclic for every $n\in\N$.
\end{cor}

\begin{proof}

We are going to check the equivalent condition given by Proposition \ref{propbirkhoffnsup}.
Bès and Peris have shown in \cite{Besper} that $T$ satisfies the Supercyclicity Criterion if and only if $(\oplus_{k=1}^{n}T)$ is supercyclic on $X^{n}$ for any $n\in\mathbb{N}$.
Let $n\in\N$, by the supercyclic version of Birkhoff Theorem, for any non-empty open sets $V,W$ in $X^n$, there exists $i\in\Z_{+}$ and $\lambda\in\K^*$ so that $(\oplus_{i=1}^nT^i)(\lambda W)\cap V\neq\emptyset$.
Let $U$ be a non-empty open set in $\Pn_n(X)$ and $V$ be a non-empty open set in $X^n$. Then, $\pi^{-1}_n(U)$ is non-empty and open in $X^n$ by definition of $\pi_n$ and for any $\lambda\in\K^*$, $\lambda\pi^{-1}_n(U)=\pi^{-1}_n(U)$.
Set $W:=\pi^{-1}_n(U)$ and use the supercyclic Birkhoff Theorem with sets $V$ and $W$, then there exists $i\in\Z_{+}$ such that $(\oplus_{i=1}^nT^i)(\pi^{-1}_n(U))\cap V\neq\emptyset$. This proves that $T$ is strongly $n$-supercyclic.
\end{proof}

Actually, one may deduce the following corollary. It is straightforward with the above corollary but we state it to justify the following remark.

\begin{cor}

Let $1\leq n<\infty$ and $X_{1},\ldots,X_{n}$ be Banach spaces and for any $i\in\{1,\ldots,n\}$, $T_{i}\in\mathcal{L}(X_{i})$. Assume that the $T_i$ satisfy the Hypercyclicity Criterion with respect to the same sequence $\{n_{k}\}_{k\in\mathbb{N}}$. Then $(\oplus_{i=1}^{n}T_{i})$ is strongly $n$-supercyclic on $X=\oplus_{i=1}^{n}X_{i}$.
\end{cor}

\begin{rem}
 One could be interested in trying to replace the Hypercyclicity Criterion above with the Supercyclicity Criterion. Feldman already proved that such operators are $n$-supercyclic \cite{Felnsupop}. We will see later in Theorem \ref{theospecerclefns} that Feldman's Theorem does not always provide strongly $n$-supercyclic operators because their spectral properties are different. In particular, this contradicts the affirmation in \cite{Shkauniv} that the operators constructed by Feldman in Example \ref{exeFelnsupop} are strongly $n$-supercyclic.
\end{rem}

\begin{rem}
As people did it for hypercyclicity, we can deduce from Proposition \ref{propbirkhoffnsup} a strong $n$-supercyclicity Criterion.
Unfortunately, this criterion is equivalent to the Hypercyclicity Criterion.
\end{rem}

\section{Some spectral properties}

It is a well-known fact for hypercyclic and supercyclic operators that the point spectrum of their adjoint is very small, in fact it counts at most one element for supercyclic operators and none for hypercyclic ones. 
Bourdon, Feldman and Shapiro proved that this was also the case for $n$-supercyclic operators, and therefore for strongly $n$-supercyclic operators, giving the following theorem:
\begin{thmbfs}{\cite{Bou}}
Suppose that $T : X \to X$ is a continuous linear operator and $n$ is a positive
integer. If $T^*$ has $n + 1$ linearly independent eigenvectors, then T is not $n$-supercyclic.
\end{thmbfs}

One can ask whether this result can be improved for strongly $n$-supercyclic operators. The following theorem shows that it is not the case. Moreover, it points out that we can choose the eigenvalues of their adjoint.
\begin{theo}\label{theonpasn-1}
Let $X$ be a complex Banach space.
 Let $\lambda_1,\ldots,\lambda_p\in\C^*$, $m_1,\ldots,m_p\in\N$ and $T$ be a bounded linear operator on $X$ and define $n=\sum_{i=1}^{p}m_i$.
Then the following assertions are equivalent:
\begin{enumerate}[(i)]
\item $S:=\oplus_{i=1}^{m_1}\lambda_1 Id\oplus\cdots\oplus_{i=1}^{m_p}\lambda_p Id\oplus T$ is strongly $n$-supercyclic on $\C^n\oplus X$;\label{tnpasn-11}
\item $\oplus_{i=1}^{m_1}\frac{T}{\lambda_{1}}\oplus\cdots\oplus_{i=1}^{m_p}\frac{T}{\lambda_{p}}$ is hypercyclic.\label{tnpasn-12}
\end{enumerate}
Moreover, in that case, $\sigma_p(S^*)=\{\lambda_1,\ldots,\lambda_p\}$ and for any $i\in\{1,\ldots,p\}$, $\lambda_i$ has multiplicity $m_i$.
\end{theo}

\begin{proof}
For the sake of convenience we denote by $\lambda_1,\ldots,\lambda_n$ the complex values we want to realise as eigenvalues of $S^*$ counted with multiplicity and let $R=\frac{T}{\lambda_{1}}\oplus\cdots\oplus\frac{T}{\lambda_{n}}$ be hypercyclic by hypothesis.
Assume that the equivalence is already proved, then the definition of $S$ implies that $\sigma_{p}(S^*)=\{\lambda_{1},\ldots,\lambda_n\}$ because  $\sigma_{p}(T^*)=\emptyset$.\\
According to the Theorem of Bourdon, Feldman and Shapiro stated above, $S$ is not strongly $k$-supercyclic for every $k<n$.

$\rhd$We begin with (\ref{tnpasn-12})$\Rightarrow$(\ref{tnpasn-11}):\\
Assume that $R$ is hypercyclic and that $(y_1,\ldots,y_n)\in X^n$ is hypercyclic for $R$, let $\{(e_{i,1},\ldots,e_{i,n})\}_{1\leq i\leq n}$ be the canonical basis of $\C^n$ and set $M=\Span\{(e_{i,1},\ldots,e_{i,n},y_i)\}_{1\leq i\leq n}$.
We are going to show that $M$ is strongly $n$-supercyclic for $S$ i.e. $\cup_{k\in\Z_{+}} \underbrace{S^k(M)\times\cdots\times S^{k}(M)}_{n\text{ times}}$ is dense in $\left(\C^n\oplus X\right)^n$. This reduces to prove:
$$\bigcup_{\substack{
            k\in\Z_{+}\\ \mu_{i,j}\in\C
           }}\left\{\sum_{i=1}^{n}\mu_{1,i}(\oplus_{j=1}^{n}\lambda_{j}^{k}e_{i,j}\oplus T^k y_i)\oplus\cdots\oplus\sum_{i=1}^{n}\mu_{n,i}(\oplus_{j=1}^{n}\lambda_{j}^{k}e_{i,j}\oplus T^k y_i)\right\}$$
is dense in $\left(\C^n\oplus X\right)^n.$\\
For this purpose, let $z=(z_{i,j})_{1\leq i\leq n,1\leq j\leq n+1}\in (\C^n\oplus X)^n$ and $\epsi>0$. We have to find $k$ and $(\mu_{i,j})_{i,j}$ in order to approach $z$ from a distance at most $\epsi$.\\
Remark that if one defines $\mu_{i,j}=\frac{z_{i,j}}{\lambda_{j}^{k}}$ we have:
$$\sum_{i=1}^{n}\mu_{1,i}(\oplus_{j=1}^{n}\lambda_{j}^{k}e_{i,j})\oplus\cdots\oplus\sum_{i=1}^{n}\mu_{n,i}(\oplus_{j=1}^{n}\lambda_{j}^{k}e_{i,j})=(z_{1,1},\ldots,z_{1,n},\ldots,z_{n,1},\ldots,z_{n,n}).$$
This leads to two cases. Either $\det((z_{i,j})_{1\leq i,j\leq n})\neq0$ and we set $x_{i,j}=z_{i,j}$ for every $1\leq i,j\leq n$. Or $\det((z_{i,j})_{1\leq i,j\leq n})=0$ and since $GL_{n}(\C)$ is dense in $M_{n}(\C)$ and $\{(e_{i,1},\ldots,e_{i,n})\}_{1\leq i\leq n}$ is a basis of $\C^n$, there exists $A:=(x_{i,j})_{1\leq i,j\leq n}\in GL_{n}(\C)$ such that:
$$\left\Vert\sum_{i=1}^{n}\frac{x_{i,j}}{\lambda_{j}^{k}}(\oplus_{j=1}^{n}\lambda_{j}^{k}e_{i,j})\oplus\cdots\oplus\sum_{i=1}^{n}\frac{x_{i,j}}{\lambda_{j}^{k}}(\oplus_{j=1}^{n}\lambda_{j}^{k}e_{i,j})-(z_{1,1},\ldots,z_{1,n},\ldots,z_{n,1},\ldots,z_{n,n})\right\Vert<\frac{\epsi}{2}.$$
In both cases, we set $\mu_{i,j}=\frac{x_{i,j}}{\lambda_{j}^{k}}$, we apply $A^{-1}$ and we need to find $k\in\Z_{+}$ so that:
$$\left\Vert \left(\begin{array}{c}
                                                       \left(\frac{T}{\lambda_1}\right)^k y_1\\
							\vdots\\
							\left(\frac{T}{\lambda_n}\right)^k y_n\\
                                                      \end{array}\right)
-A^{-1}\left(\begin{array}{c}
                                                       z_{1,n+1}\\
							\vdots\\
							z_{n,n+1}\\
                                                      \end{array}\right)\right\Vert<\frac{\epsi}{2\Vert A\Vert}.$$
But such a $k\in\Z_{+}$ exists because $(y_1,\ldots,y_n)$ is hypercyclic for $R$. Since we found $k\in\Z_{+}$ and $(\mu_{i,j})_{i,j}$ such that:
$$\left\Vert\sum_{i=1}^{n}\mu_{1,i}(\oplus_{j=1}^{n}\lambda_{j}^{k}e_{i,j}\oplus T^k y_i)\oplus\cdots\oplus\sum_{i=1}^{n}\mu_{n,i}(\oplus_{j=1}^{n}\lambda_{j}^{k}e_{i,j}\oplus T^k y_i)-(z_{i,j})_{1\leq i,j\leq n}\right\Vert<\frac{\epsi}{2}+\frac{\epsi}{2}=\epsi,$$
then $S$ is strongly $n$-supercyclic.

$\rhd$(\ref{tnpasn-11})$\Rightarrow$(\ref{tnpasn-12}):\\
Assume that $S$ is strongly $n$-supercyclic and let $M$ be a strongly $n$-supercyclic subspace for $S$ and denote by $M_0$ its projection on $\C^n$. Then, $M_0$ is strongly $\dim(M_0)$-supercyclic for $S_{\vert\C^n}$ and $\C^n$ being of dimension $n$, the main result of Bourdon, Feldman and Shapiro (\cite{Bou}) implies that $\dim(M_0)=n$, i.e. $M_0=\C^n$.
Thus, it is possible to choose a basis of $M$ like the following: 
$$M=\Span\left(\left(\begin{array}{c}
                                                                          1\\0\\\vdots\\0\\x_1
                                                                         \end{array}\right),\left(\begin{array}{c}
                                                                          0\\1\\\vdots\\0\\x_2
                                                                         \end{array}\right),\cdots,\left(\begin{array}{c}
                                                                          0\\0\\\vdots\\1\\x_n
                                                                         \end{array}\right)\right).$$
Let us prove that $R$ is hypercyclic.\\
Let $(z_1,\cdots,z_n)\in X^n$.
Since $S$ is strongly $n$-supercyclic, there exists a strictly increasing sequence $(n_k)_{k\in\Z_{+}}$ and complex numbers $(\mu_{i,j}^{(n_k)})_{1\leq i,j\leq n}$ such that for every $i\in\{1,\ldots,n\}$:
$$\begin{cases}
 \mu_{i,i}^{(n_k)}\lambda_{i}^{n_k}\underset{k\to+\infty}{\longrightarrow}1,\\
\mu_{i,j}^{(n_k)}\lambda_{j}^{n_k}\underset{k\to+\infty}{\longrightarrow}0\text{ for any }j\neq i,\\
z_{i}^{(n_k)}:=\sum_{j=1}^{n}\mu_{i,j}^{(n_k)}T^{n_k}x_j\underset{k\to+\infty}{\longrightarrow} z_i.
\end{cases}
$$
Take also, $$A^{(n_k)}=\left(\begin{array}{ccc}
                          \mu_{1,1}^{(n_k)}\lambda_{1}^{n_k}&\cdots&\mu_{n,1}^{(n_k)}\lambda_{n}^{n_k}\\
			  \vdots&\ddots&\vdots\\
			  \mu_{1,n}^{(n_k)}\lambda_{1}^{n_k}&\cdots&\mu_{n,n}^{(n_k)}\lambda_{n}^{n_k}\\
                         \end{array}\right).$$
Obviously, with the preceding convergences, $A^{(n_k)}\underset{k\to+\infty}{\longrightarrow} Id$, so we may suppose that $A^{(n_k)}$ is invertible and thus $\left(A^{(n_k)}\right)^{-1}\underset{k\to+\infty}{\longrightarrow}Id$ too.\\
Then, remark that the previous system is equivalent to the following: 
%
%
$$\left(\begin{array}{c}
                                                      \frac{T^{n_k}x_1}{\lambda_{1}^{n_k}}\\
\vdots\\
\frac{T^{n_k}x_n}{\lambda_{n}^{n_k}}

                                                     \end{array}\right)=\left(A^{(n_k)}\right)^{-1}\left(\begin{array}{c}
                                                      z_{1}^{(n_k)}\\
\vdots\\
z_n^{(n_k)}

                                                     \end{array}\right)\underset{k\to+\infty}{\longrightarrow}\left(\begin{array}{c}
                                                      z_{1}\\
\vdots\\
z_n

                                                     \end{array}\right).$$
This proves the hypercyclicity of $R$.
\end{proof}

Feldman showed in \cite{Felnsupop} that there exists operators that are $n$-supercyclic but not $(n-1)$-supercyclic. The last result allows us to give an example of a strongly $n$-supercyclic operator which is not strongly $k$-supercyclic, for every $k<n$.

\begin{exe}\label{exepassup}
Let $B$ be the classical backward shift on $\ell^2(\Z_{+})$ being defined by $B(a_0,a_1,a_2,\ldots)=(a_1,a_2,\ldots)$ and let also $\lambda_1,\ldots,\lambda_n\in\mathbb{D}$. Then, the operator defined on $\C^n\oplus\ell^2(\Z_{+})$ by $T=\lambda_1 Id\oplus\cdots\oplus\lambda_n Id\oplus B$ is strongly $n$-supercyclic but not strongly $k$-supercyclic for every $k<n$.
\\Indeed, a classical result says that $\frac{B}{\lambda}$ satisfies the Hypercyclicity Criterion for the whole sequence of integers if and only if $\vert\lambda\vert<1$. Thus, $\frac{B}{\lambda_1}\oplus\cdots\oplus\frac{B}{\lambda_n}$ satisfies also the Hypercyclicity Criterion and $T$ is strongly $n$-supercyclic by Theorem \ref{theonpasn-1}. Nevertheless, the fact that $T$ is not $k$-supercyclic for $k<n$ is clear because if it was, then the restriction of $T$ to $\C^n$ would also be $k$-supercyclic but this contradicts Theorem \ref{theoboufelshap} \cite{Bou}.
\end{exe}

Since strongly $n$-supercyclic operators are in particular $n$-supercyclic, they inherit their spectral properties, hence the Circle Theorem applies to these ones. Therefore, for every strongly $n$-supercyclic operator, there exists a set of at most $n$ circles intersecting every component of the spectrum of $T$. This was obtained by Feldman \cite{Felnsupop} for $n$-supercyclic operators and he provided also examples for which $n$ circles were necessary.
In the case of strongly $n$-supercyclic operators, we are able to improve the Circle Theorem:

\begin{theo}\label{theospecerclefns}
Assume that $X$ is a complex Banach space and $T$ is a strongly $n$-supercyclic operator on $X$.\\
Then we can decompose $X=F\oplus X_0$, where $F$ and $X_0$ are $T$-invariant, $F$ has dimension at most $n$ and there exists $R\geq 0$ such that the circle $\{z\in\C: \vert z\vert=R\}$ intersects every component from the spectrum of $T_0:=T_{\vert X_0}$.
\\Moreover, in the particular case $n=2$, $T_{\vert F}$ is a diagonal operator.
\end{theo}

\begin{proof}

The theorem is trivial if there already exists a circle intersecting all the components from the spectrum of $T$.

If such a circle does not exist, then there exist $R\geq 0$ and two components $C_1,C_2$ from $\sigma(T)$ such that $C_1\subset B(0,R)$ and $C_2\subset\C\setminus\overline{B(0,R)}$.
Upon considering a scalar multiple of $T$, one may suppose that $R=1$.
Thus $\sigma(T)=\sigma_{1}\cup\sigma_2\cup\sigma_3$ where $\sigma_1\subset\mathbb{D}$, $\sigma_2\subset\C\setminus\overline{\mathbb{D}}$ and $\sigma_1,\sigma_2,\sigma_3$ are closed and pairwise disjoint.
Then, thanks to Riesz Theorem \cite{Bay} one can write $T=T_1\oplus T_2\oplus T_3$ on $X=X_1\oplus X_2\oplus X_3$ where $\sigma(T_i)=\sigma_i$ for $i=1,2,3$.
\\We are first going to prove that $\dim(X_1)\leq n-1$.
Assume to the contrary that $\dim(X_1)\geq n$.\\
Then, one can choose $(u_1,\ldots,u_n)\in X_{1}^{n}$ such that for every $i\in\{1,\ldots,n\}$, $\Vert z-u_i\Vert>1$ for any $z\in\Span(u_1,\ldots,u_{i-1},u_{i+1},\ldots,u_n)$.
\\Let $L=\small\Span\left(\left(\begin{array}{c}
                            x_1\\
y_1\\
z_1\\
                           \end{array}
\right),\cdots,\left(\begin{array}{c}
                            x_n\\
y_n\\
z_n\\
                           \end{array}
\right)\right)$ be a strongly $n$-supercyclic subspace for $T$, $(n_k)_{k\in\Z_{+}}$ be a strictly increasing sequence and $A_k\in M_n(\C)$ such that:
$$A_k\left(\begin{array}{c}
                            T_{1}^{n_k}x_1\\
\vdots\\
T_{1}^{n_k}x_n\\
                           \end{array}
\right)\underset{k\to+\infty}{\longrightarrow}\left(\begin{array}{c}
                            u_1\\
\vdots\\
u_n\\
                           \end{array}
\right)\text{ and }A_k\left(\begin{array}{c}
                            T_{2}^{n_k}y_1\\
\vdots\\
T_{2}^{n_k}y_n\\
                           \end{array}
\right)\underset{k\to+\infty}{\longrightarrow}\left(\begin{array}{c}
                            0\\
\vdots\\
0\\
                           \end{array}
\right).
$$
In addition, by density of $GL_{n}(\C)$ in $M_n(\C)$, one can assume that $A_k$ is invertible for every $k\in\Z_{+}$.
This yields:$$\left(\begin{array}{c}
                            T_{1}^{n_k}x_1\\
\vdots\\
T_{1}^{n_k}x_n\\
                           \end{array}
\right)=A_{k}^{-1}\left(\begin{array}{c}
                            u_{1,k}\\
\vdots\\
u_{n,k}\\
                           \end{array}
\right):=A_{k}^{-1}\left(\begin{array}{c}
                            u_{1}+\epsi_{1,k}\\
\vdots\\
u_{n}+\epsi_{n,k}\\
                           \end{array}
\right)$$where for every $i\in\{1,\ldots,n\}$, $\Vert\epsi_{i,k}\Vert\underset{k\to+\infty}{\longrightarrow}0.$\\
Set $A_{k}^{-1}=\left(\begin{array}{ccc}
                    a_{1,1}^{k}&\cdots&a_{1,n}^{k}\\
\vdots&\ddots&\vdots\\
a_{n,1}^{k}&\cdots&a_{n,n}^{k}\\
                   \end{array}\right)$. Since $A_k\left(\begin{array}{c}
                            T_{2}^{n_k}y_1\\
\vdots\\
T_{2}^{n_k}y_n
                           \end{array}
\right)\underset{k\to+\infty}{\longrightarrow}\left(\begin{array}{c}
                            0\\
\vdots\\
0\\
                           \end{array}
\right)$ and $\sigma(T_{2})\subset \C\setminus\overline{\mathbb{D}}$, it follows that $\Vert A_{k}^{-1}\Vert\underset{k\to+\infty}{\longrightarrow}+\infty$ hence $\max(\vert a_{i,j}^{k}\vert)_{1\leq i,j\leq n}\underset{k\to+\infty}{\longrightarrow}+\infty$ and thus for any $m\in\{1,\ldots,n\}$, $\frac{T_{1}^{n_k}x_m}{\max(\vert a_{i,j}^{k}\vert)_{1\leq i,j\leq n}}\underset{k\to+\infty}{\longrightarrow}0$.
Let $k\in\Z_{+}$ be such that for any $m\in\{1,\ldots,n\}$, $\frac{\Vert T_{1}^{n_k}x_m\Vert}{\max(\vert a_{i,j}^{k}\vert)_{1\leq i,j\leq n}}<\frac{1}{2}$ and $\Vert\epsi_{m,k}\Vert<\frac{1}{2}$ and set $\vert a_{p,q}^{k}\vert:=\max(\vert a_{i,j}^{k}\vert)_{1\leq i,j\leq n}$.
Then, we have $T_{1}^{n_k}x_p=\sum_{i=1}^{n}a_{p,i}^{k}u_{i,k}$, yielding $$\left\Vert u_{q,k}+\sum_{i=1,\ i\neq q}^{n}\frac{a_{p,i}^{k}}{a_{p,q}^{k}}u_{i}^{k}\right\Vert=\left\Vert\frac{T_{1}^{n_k}x_p}{a_{p,q}^{k}}\right\Vert< \frac{1}{2}.$$
This result contradicts our first assumption that for every $i\in\{1,\ldots,n\}$, $\Vert z-u_i\Vert>1$ for any $z\in\Span(u_1,\ldots,u_{i-1},u_{i+1},\ldots,u_n)$.
Hence $\dim(X_1)\leq n-1$ and if $n=2$, we get $\dim(X_1)=1$.
\\We can do the same process with $T_2\oplus T_3$ which is strongly $n$-supercyclic thus either there exists a circle intersecting every component of the spectrum of $T_2\oplus T_3$ and the proof is finished, or we can decompose $T_2\oplus T_3$ as a direct sum of two operators where the first one is defined on a space of dimension lower than $n-1$.
Then, as there is an at most $n$-dimensional subspace in this decomposition because there is no strongly $n$-supercyclic operators on a space of dimension strictly greater than $n$ according to Theorem \ref{theoboufelshap}. Thus, we can iterate this process only a finite number of times.
This proves the first part of the theorem. The particular case $n=2$ part, is clear from the proof. 

\end{proof}

In particular, considering $n=2$ in the preceding theorem gives an alternative generalising the case of supercyclic operators. Indeed, for a supercyclic operator it is well-known that the point spectrum is either empty or a singleton $\{\lambda\}$ and in the last case, $\lambda^{-1}T$ is hypercyclic on an hyperplane of $X$. The following corollary gives a similar result for strongly 2-supercyclic operators.
\begin{cor}\label{corspef2s}
Assume that $X$ is a complex Banach space and $T$ is a strongly 2-supercyclic operator on $X$. Then, one of the following properties applies:

$\bullet$ There exists $R\geq 0$ such that the circle $\{z\in\C: \vert z\vert=R\}$ intersects every component from the spectrum of $T$,

$\bullet$ $T=\left(\begin{array}{cc}
a &0  \\ 
0 &S  \\
\end{array}\right)$ with $S$ being a supercyclic operator, $a\in\C^{*}$

$\bullet$ $T=\left(\begin{array}{ccc}
a &0 &0 \\ 
0 &b &0 \\
0 &0 &S \\
\end{array}\right)$ with $\frac{S}{a}\oplus\frac{S}{b}$ hypercyclic, $a,b\in\C^{*}$.
\end{cor}

\begin{proof}
According to Theorem \ref{theospecerclefns} we have the following alternative: either there exists a circle intersecting every component of the spectrum of $T$ or we can decompose $X=F\oplus X_0$ with $F$ and $X_0$, $F$ being of dimension at most 2 and $S:=T_{\vert F}$ being diagonal and there exists a circle intersecting every component of the spectrum of $T_0:=T_{\vert X_0}$.
\\$\bullet$  If $\dim(F)=1$ then $T=\left(\begin{array}{cc}
a &0  \\ 
0 &S \\
\end{array}\right)$ for some $a\in\C^{*}$.\\
We can suppose that $L$ is a strongly 2-supercyclic subspace and that $L=\Span\left(\left(\begin{array}{c}
1  \\ 
x \\
\end{array}\right),\left(\begin{array}{c}
0  \\ 
y \\
\end{array}\right)\right)$.
Let $z\in X_0$.
Since $T$ is strongly 2-supercyclic, there exists an increasing sequence $(n_k)_{k\in\Z_{+}}$ and $A_k:=\left(\begin{array}{cc}
\lambda_{1}^{n_k}&\lambda_{2}^{n_k}  \\ 
 \mu_{1}^{n_k}&\mu_{2}^{n_k} \\
\end{array}\right)\in GL_{2}(\C)$ such that
$$A_k\left(\begin{array}{c}
S^{n_k}x  \\ 
S^{n_k}y \\
\end{array}\right)=\left(\begin{array}{c}
0+\epsi_{1}^{n_k}  \\ 
z +\epsi_{2}^{n_k}\\
\end{array}\right)\text{ where }\epsi_{1}^{n_k}\underset{k\to+\infty}{\longrightarrow}0\text{ and }\epsi_{2}^{n_k}\underset{k\to+\infty}{\longrightarrow}0$$ and 
$$A_k\left(\begin{array}{c}
a^k  \\ 
0 \\
\end{array}\right)=\left(\begin{array}{c}
1+\delta_{1}^{n_k}  \\ 
0 +\delta_{2}^{n_k}\\
\end{array}\right)\text{ where }\delta_{1}^{n_k}\underset{k\to+\infty}{\longrightarrow}0\text{ and }\delta_{2}^{n_k}\underset{k\to+\infty}{\longrightarrow}0.$$
Thus, considering the inverse of $A_k$ we get: $S^{n_k}y=\frac{-\mu_{1}^{n_k}\epsi_{1}^{n_k}+\lambda_{1}^{n_k}(z+\epsi_{2}^{n_k})}{\lambda_{1}^{n_k}\mu_{2}^{n_k}-\lambda_{2}^{n_k}\mu_{1}^{n_k}}$. Multiply the last equality by $a^k$ to obtain:
$$\begin{aligned}
a^k(\lambda_{1}^{n_k}\mu_{2}^{n_k}-\lambda_{2}^{n_k}\mu_{1}^{n_k})S^{n_k}y&=-a^k\mu_{1}^{n_k}\epsi_{1}^{n_k}+a^k\lambda_{1}^{n_k}(z+\epsi_{2}^{n_k})&\\
&=-\delta_{2}^{n_k}\epsi_{1}^{n_k}+(1+\delta_{1}^{n_k})(z+\epsi_{2}^{n_k})&\underset{k\to+\infty}{\longrightarrow}z.\\ 
\end{aligned}$$
Hence $S$ is supercyclic on $X_0$.\\
$\bullet$  If $\dim(F)=2$ then $T=\left(\begin{array}{ccc}
a &0 &0 \\ 
0 &b &0 \\
0 &0 &S \\
\end{array}\right)$, for some $a,b\in\C^{*}$.\\
It suffices to apply Theorem \ref{theonpasn-1} to conclude that $\frac{S}{a}\oplus\frac{S}{b}$ is hypercyclic.
\end{proof}

\begin{rem}
Actually, these three conditions are necessary, and we give an example for each one.

The first point is easy, simply consider an operator satisfying the Supercyclicity Criterion: the circle exists because the operator is supercyclic and it is strongly 2-supercyclic thanks to Corollary \ref{corpremex}.
 
The second one is trickier:
let $\phi\in H^{\infty}(\mathbb{D})$ be defined by $\phi(z)=1+\imath+z$, and let us denote by $M_\phi$ the multiplication operator associated to $\phi$ on $H^{2}(\mathbb{D})$. Set also $R_{n}:=\sum_{i=1}^{n-1}(M_{\phi}^{*})^{i}$. Then, one may prove following Exercise 1.9 in \cite{Bay} that there exists a universal vector for $R_n$: $u\in H^{2}(\mathbb{D})$ and $u\notin Im(M_{\phi}^{*}-I)$ and that 
$\left(\begin{array}{cc}
1&0\\
u&M_{\phi}^{*}\\
                                                   \end{array}\right)$ is supercyclic and is not similar to an operator of the form $I\oplus S$. Noticing also that $\{R_n\oplus(M_{\phi}^{*})^{n}\}_{n\geq2}$ satisfies the Universality Criterion, then one can prove that 
$T:=\left(\begin{array}{ccc}
a&0&0\\
0&1&0\\
0&u&M_{\phi}^{*}\\
\end{array}\right)$ is strongly 2-supercyclic on $\C^2\oplus H^{2}(\mathbb{D})$ and is not similar to any operator of the shape $bI\oplus cI\oplus T_0$ and does not even admit a circle intersecting every component of its spectrum for a well-chosen complex number $a$.

Finally, the third case is simple: $T=\left(\begin{array}{ccc}
                                                    -1&0&0\\
0&-\frac{1}{2}&0\\
0&0&M_{\phi}^{*}\\
                                                   \end{array}\right)$ is strongly 2-supercyclic on $\C^2\oplus H^{2}(\mathbb{D})$ with $\phi(z)=1+z$ by Theorem \ref{theonpasn-1} but its spectrum is $\sigma(T)=\left\{-1,-\frac{1}{2}\right\}\cup D(1,1)$.
\end{rem}

\section{Other classes of interesting examples}

Until now, we proved several properties of strongly $n$-supercyclic operators and we came across different classes of examples but links between strong $(n-1),n,(n+1)$-supercyclic operators are not well understood yet. This part provides some answers but also some interesting questions on the subject.

\subsection{A class of strongly $k$-supercyclic operators with $k\geq n$}

The following example generalises Corollary \ref{corpremex}. It has been proved in \cite{Bou} and \cite{Ernstnsupstrongnsupop} that strong $n$-supercyclicity is purely infinite dimensional. We are going to make use of this fact to construct an operator being strongly $k$-supercyclic if and only if $k\geq n$.
  
\begin{exe}
Assume that $S$ satisfies the Hypercyclicity Criterion on a Banach space $Y$ and define $T=Id\oplus S$ on $X=\K^{n}\oplus Y$.
Then $T$ is strongly $k$-supercyclic if and only if $k\geq n$.
\end{exe}

\begin{proof}
$\ $\\

$\bullet$ We first prove that if $T$ is strongly $k$-supercyclic, then $k\geq n$. Assume to the contrary that $k<n$, then restricting $T$ to $\K^n$, one obtains that $Id$ is strongly $k$-supercyclic on $\K^n$ with $k<n$. This is impossible by \cite{Bou} for the complex case and \cite{Ernstnsupstrongnsupop} for the real case.

$\bullet$ Let us prove now that for every $p\geq n$, $T$ is strongly $p$-supercyclic.

The following lemma is the key of the proof.

\begin{lem}\label{lemfornpr}
Let $p\geq 1$. Then, there exists $(y_1,\ldots,y_p)\in Y^p$ such that for any $A=(\lambda_{i,j})_{1\leq i,j\leq p}\in GL_p(\K)$, the set: $$\left\{S^k\left(\sum_{i=1}^{p}\lambda_{1,i}y_i\right)\oplus\cdots\oplus S^{k}\left(\sum_{i=1}^{p}\lambda_{p,i}y_i\right)\right\}_{k\in\Z_{+}}\text{ is dense in }Y^{p}.$$ 
\end{lem}

\begin{proof}

Since $S$ satisfies the Hypercyclicity Criterion, then $L:=\underbrace{S\oplus\cdots\oplus S}_{p \text{ times}}$ is hypercyclic too \cite{Besper}.
Let $(y_1,\ldots,y_p)\in Y^{p}$ is a hypercyclic vector for $L$.
Since $(y_1,\ldots,y_p)\in Y^{p}$ is hypercyclic for $L$. Then $(S^{k}(y_{1}),\ldots,S^{k}(y_{p}))_{k\geq0}$ is dense in $Y^{p}$.
Since $A$ is invertible, some simple computations imply the result of the lemma.
%
\end{proof}

We come back to the proof of the example.\\
Let $p\geq n$ and $\{e_1,\ldots,e_p\}$ be a generating family of $\K^{n}$ with $p$ elements and $(y_1,\ldots,y_p)$ given by the previous lemma and denote $x_i=(e_i,y_i)\in X$, for every $i\in\{1,\ldots,p\}$.\\
It is easy to show that $M:=\Span(x_1,\ldots,x_p)$ is strongly $p$-supercyclic for $T$.
Actually, it suffices to prove that $\cup_{k\in\Z_{+}}\underbrace{T^k(M)\times\cdots\times T^{k}(M)}_{p\text{ times}}$ is dense in $X^{p}$ thanks to Proposition \ref{propbeq}.\\
The use of the definition of $M$ reduces the proof to the following assertion:
$$\cup_{k\in\Z_{+}, (\lambda_{i,j})_{1\leq i,j\leq p}\in M_p(\K)}\sum_{i=1}^{p}\lambda_{1,i}(e_i\oplus S^k y_i)\oplus\cdots\oplus\sum_{i=1}^{p}\lambda_{p,i}(e_i\oplus S^k y_i)\text{ is dense in }X^{p}.$$
For this purpose, let $\epsi>0$, $(t_1,\ldots,t_p)\in(\K^{n})^p$ and $(z_1,\ldots,z_p)\in Y^p$.
Since $GL_p(\K)$ is dense in $M_p(\K)$, there exists $A=(\lambda_{i,j})_{1\leq i,j\leq p}\in GL_p(\K)$ so that:
$$\left\Vert A\left(\begin{array}{c}e_1\\ \vdots\\ e_p\\ \end{array}\right)-\left(\begin{array}{c}t_1\\ \vdots\\ t_p\\ \end{array}\right)\right\Vert<\frac{\epsi}{2}.$$
On the other hand, Lemma \ref{lemfornpr} implies that there is $k\in\Z_{+}$ satisfying:
$$\left\Vert S^k\left(\sum_{i=1}^{p}\lambda_{1,i}y_i\right)\oplus\cdots\oplus S^k\left(\sum_{i=1}^{p}\lambda_{p,i}y_i\right)-\left(z_{1},\ldots,z_p\right)\right\Vert\leq \frac{\epsi}{2}.$$
Hence, 
$$\left\Vert\sum_{i=1}^{p}\lambda_{1,i}\left(e_i\oplus S^k y_i\right)\oplus\cdots\oplus\sum_{i=1}^{p}\lambda_{p,i}\left(e_i\oplus S^k y_i\right)-\oplus_{i=1}^{p}\left(t_i\oplus z_i\right)\right\Vert<\epsi$$
This is the relation we were looking for. Thus, $T$ is strongly $p$-supercyclic.
\end{proof}

\begin{rem}
In the same spirit, one may easily prove that strongly $n$-supercyclic operators given by Theorem \ref{theonpasn-1} are not strongly $k$-supercyclic for $k<n$. \\
Building on the same ideas as in the previous example but using properties of rotations, one may easily construct an operator being $n$-supercyclic from a particular rank and being strongly $n$-supercyclic from a strictly greater rank.
\end{rem}

\subsection{A supercyclic operator which is not strongly $n$-supercyclic for a fixed $n\geq2$}

We already noticed in the previous part that strong $n$-supercyclicity does not imply strong $(n-1)$-supercyclicity. It is a natural question to ask whether the contrary is true or not: does strong $n$-supercyclicity imply strong $(n+1)$-supercyclicity?
In the following, we prove that it is not the case for $n=1$. To do so, we will construct a supercyclic operator which is not strongly $p$-supercyclic for $p\geq2$.
Operators satisfying the Supercyclicity Criterion are useless in this context because we noticed in Corollary \ref{corpremex} that these operators are strongly $n$-supercyclic for any $n\geq1$.
Thus, we are forced to consider operators that are less handy. Actually, we are modifying the construction of a hypercyclic operator which is not weakly mixing from Bayart and Matheron \cite{Bay} to achieve it.

\begin{theo}\label{theosupnonpsup}
 Assume that $X$ is a Banach space with an unconditional normalised basis $(e_{i})_{i\in\Z_{+}}$ for which the associated forward shift $(e_{i})_{i\in\Z_{+}}$ is continuous and let $p\geq2$. Then, there exists a supercyclic operator which is not strongly $h$-supercyclic for any $2\leq h\leq p$.
\end{theo}

The proof of this theorem is long and is based on the work of Bayart and Matheron \cite[section 4.2]{Bay}. The proof is a succession of intermediate results leading to the final proof.
The main idea is to construct an operator and to create a criterion to check that this operator is not strongly $h$-supercyclic. We may refer the reader to the book \cite{Bay} for certain proofs.
\\Let us define some material we need in the sequel.

Assume that $T$ is a linear bounded operator on a topological vector space $X$ and let $e_0\in X$, then we set:
$$\begin{aligned}
\K[T](e_{0})&=\{P(T)(e_{0}), P\in\K[X]\}&\\
&=\Span\{T^{i}(e_{0}),i\in\Z_{+}\}.&\\ 
\end{aligned}$$
Since two polynomials with variable $T$ always commute, we can also define a product on $\K[T](e_{0})$ by: $$P(T)e_{0}\cdot Q(T)e_{0}=PQ(T)e_{0}.$$
We first give a technical lemma proving the convergence of a sequence of unit spheres if there is a sequence of  basis converging to another basis.
\begin{lem}\label{leminfcercle} 
Assume that $X$ is a normed vector space, $h\geq2$, and that $E=\Span(u_1,\ldots,u_h)$ is a subspace of dimension $h$. For every $1\leq i\leq h$, let $(v_{i}^{n})_{n\in\Z_{+}}$ be a sequence of elements of $X$ such that $\Vert v_{i}^{n}-u_i\Vert\leq\frac{1}{n}$ and set $F_n=\Span(v_{1}^{n},\ldots,v_{h}^{n})$.
Then, $\underset{z\in F_n,\Vert z\Vert=1}{\sup}\underset{x\in E,\Vert x\Vert=1}{\inf}\Vert x-z\Vert\underset{n\to+\infty}{\longrightarrow}0$.
\end{lem}

\begin{proof}
Let $n\geq N$, and $T_{n}: X\to X$ be defined by $T_{n}(x)=\sum_{i=1}^{h}u_{i}^{*}(x)(u_i-v_{i}^{n})$ for every $1\leq i\leq h$.
Remark that:
$$\Vert T_{n}\Vert=\underset{\Vert x\Vert=1}{\sup}\left\Vert \sum_{i=1}^{h}u_{i}^{*}(x)(u_i-v_{i}^{n})\right\Vert\leq\sum_{i=1}^{h}\Vert u_{i}^{*}\Vert\Vert u_i-v_{i}^{n}\Vert\leq \frac{Mh}{n}\leq \frac{1}{2}.$$
Thus, it follows that the operator $S_{n}:=I-T_{n}: X\to X$ satisfies $S_{n}(u_i)=v_{i}^{n}$ for any $1\leq i\leq h$ and is invertible with $S_{n}^{-1}=\sum_{i=0}^{+\infty}T_{n}^{i}$.
Moreover, we deduce the following upper bound for $S_{n}^{-1}$:
$$\Vert S_{n}^{-1}\Vert=\left\Vert \sum_{i=0}^{+\infty}T_{n}^{i}\right\Vert\leq \sum_{i=0}^{+\infty}\Vert T_{n}\Vert^i\leq \sum_{i=0}^{+\infty}\frac{1}{2^i}=2.$$
Set $z_{n}=a_{1}^{n}v_{1}^{n}+\ldots+a_{h}^{n}v_{h}^{n}\in F_{n}$ with $\Vert z_{n}\Vert=1$. Then,  we notice that $\Vert S_{n}^{-1}(z_{n})\Vert=\Vert a_{1}^{n}u_{1}+\ldots+a_{h}^{n}u_{h}\Vert\leq2$. Thus, $u_{1},\ldots,u_{h}$ being a linearly independent family, we deduce that the sequences of coefficients $(a_{i}^{n})_{n\in\N}$ are bounded and upon passing to a subsequence, we can suppose that $z_{n}$ converges to some vector on the unit sphere of $E$. This proves the lemma.
\end{proof}

To provide an operator as claimed in Theorem \ref{theosupnonpsup}, we need to be able to check the non-strong $h$-supercyclicity for the operator $T$ we are going to construct. The following lemma gives such a criterion.

\begin{lem}\label{lemcritnfps}
Assume that $X$ is a topological vector space and $T\in L(X)$ is cyclic with cyclic vector $e_{0}$ and set $h\geq2$. Assume that there exist $2h$ linear forms $\Phi_\delta:\K[T](e_0)\to \K$ for $0\leq \delta\leq 2h-1$, such that the maps $(P(T)e_{0},Q(T)e_{0})\mapsto \Phi_\delta((PQ)(T)e_{0})$ are continuous on $\K[T](e_0)\times \K[T](e_0)$ and satisfying for every $0\leq \delta \leq 2h-1$
$$\Phi_\delta(T^{\delta}e_{0})=1\text{ and }\Phi_\delta(T^{i}e_{0})=0\text{ for }0\leq i\neq\delta\leq 2h-1.$$
Then, $T$ is not strongly $h$-supercyclic.
\end{lem}

\begin{rem}
In particular, $T$ does not satisfy the Supercyclicity Criterion.
\end{rem}

\begin{proof}
Assume that $T$ is strongly $h$-supercyclic on $X$ and $n\in\N$.
Set also 
\begin{equation}\label{eqrap}
 E=\Span(e_{0},\ldots,T^{h-1}e_{0})\text{, and }E_{n}\in\pi_{h}\left(B\left((e_{0},\ldots,T^{h-1}e_{0});\frac{1}{n}\right)\right)\cap\mathcal{E}S_{h}(T).
\end{equation}
Thus, there exists $m_{n}\in\mathbb{N}, x_{n},y_{n}\in E_{n}$ linearly independent such that: $T^{m_{n}}x_{n}\in B\left(e_{0};\frac{1}{2n}\right)$ and $T^{m_{n}}y_{n}\in B\left(T^{h}e_{0};\frac{1}{2n}\right)$.
Moreover, fix $\varepsilon_{n}=\min\left(\frac{1}{n},\frac{1}{2n\Vert T^{m_{n}}\Vert},\frac{\Vert x_{n}\Vert}{2n+1},\frac{\Vert y_{n}\Vert}{2n+1}\right)$, then $e_{0}$ being cyclic for $T$, there exists $P_{n}, Q_{n}\in\mathbb{K}[X]$ such that:\\
$$P_{n}(T)e_{0}\in B(x_{n};\varepsilon_{n})\text{ and }Q_{n}(T)e_{0}\in B(y_{n};\varepsilon_{n}).$$
Thus, $$T^{m_{n}}(P_{n}(T)e_{0})\in B\left(e_{0};\frac{1}{n}\right)\text{ and }T^{m_{n}}(Q_{n}(T)e_{0})\in B\left(T^{h}e_{0};\frac{1}{n}\right).$$
Pick also $a_{0}^{n}e_{0}+\ldots+a_{h-1}^{n}T^{h-1}e_{0}\in E$ such that: 
$$\begin{array}{c}
\Vert a_{0}^{n}e_{0}+\ldots+a_{h-1}^{n}T^{h-1}e_{0}\Vert=1\\
\text{ and }\\
\left\Vert a_{0}^{n}e_{0}+\ldots+a_{h-1}^{n}T^{h-1}e_{0}-\frac{P_{n}(T)e_{0}}{\Vert P_{n}(T)e_{0}\Vert}\right\Vert=\underset{x\in E, \Vert x\Vert=1}{\inf}\left\Vert x-\frac{P_{n}(T)e_{0}}{\Vert P_{n}(T)e_{0}\Vert}\right\Vert.
\end{array}$$
Then, $$\underset{x\in E, \Vert x\Vert=1}{\inf}\left\Vert x-\frac{P_{n}(T)e_{0}}{\Vert P_{n}(T)e_{0}\Vert}\right\Vert\underset{n\to+\infty}{\longrightarrow}0.$$
Let us prove this last point. First, we split the norm:
$$\inf_{x\in E, \Vert x\Vert=1}\left\Vert x-\frac{P_{n}(T)e_{0}}{\Vert P_{n}(T)e_{0}\Vert}\right\Vert\leq \inf_{x\in E, \Vert x\Vert=1}\left\Vert x-\frac{x_{n}}{\Vert x_{n}\Vert}\right\Vert+\left\Vert\frac{x_{n}}{\Vert x_{n}\Vert}-\frac{P_{n}(T)e_{0}}{\Vert P_{n}(T)e_{0}\Vert}\right\Vert$$
and it suffices to prove that each part tends to 0 when $n$ grows.
In fact, the first convergence to 0 is given by Lemma \ref{leminfcercle} (with $u_{i}= T^{i-1}(e_{0})$ and $(v_{1}^{n},\ldots,v_{h}^{n})=$basis of $E_n$ given by (\ref{eqrap})) and let us deal with the second part:
$$\begin{aligned}
\left\Vert\frac{x_{n}}{\Vert x_{n}\Vert}-\frac{P_{n}(T)e_{0}}{\Vert P_{n}(T)e_{0}\Vert}\right\Vert&=\left\Vert\frac{x_{n}}{\Vert x_{n}\Vert}-\frac{x_{n}}{\Vert P_{n}(T)e_{0}\Vert}+\frac{x_{n}}{\Vert P_{n}(T)e_{0}\Vert}-\frac{P_{n}(T)e_{0}}{\Vert P_{n}(T)e_{0}\Vert}\right\Vert&\\
&\leq\Vert x_{n}\Vert \left\vert\frac{1}{\Vert x_{n}\Vert}-\frac{1}{\Vert P_{n}(T)e_{0}\Vert}\right\vert+\frac{\Vert x_{n}-P_{n}(T)e_{0}\Vert}{\Vert P_{n}(T)e_{0}\Vert}&\\
&\leq \frac{\varepsilon_{n}}{\Vert x_{n}\Vert -\varepsilon_{n}}+\frac{\varepsilon_{n}}{\Vert x_{n}\Vert-\varepsilon_{n}}\leq \frac{1}{n}\text{ by definition of }\varepsilon_{n}.\\
\end{aligned}$$
Thus we have the expected convergence.\\
Doing the same thing, we also pick $b_{0}^{n}e_{0}+\ldots+b_{h-1}^{n}T^{h-1}e_{0}\in E$ such that: $$\Vert b_{0}^{n}e_{0}+\ldots+b_{h-1}^{n}T^{h-1}e_{0}\Vert=1$$ and $$\left\Vert b_{0}^{n}e_{0}+\ldots+b_{h-1}^{n}T^{h-1}e_{0}-\frac{Q_{n}(T)e_{0}}{\Vert Q_{n}(T)e_{0}\Vert}\right\Vert=\underset{x\in E, \Vert x\Vert=1}{\inf}\left\Vert x-\frac{Q_{n}(T)e_{0}}{\Vert Q_{n}(T)e_{0}\Vert}\right\Vert \underset{n\to+\infty}{\longrightarrow}0.$$
Moreover, extracting an appropriate strictly increasing subsequence $(s_{k})_{k\in\mathbb{N}}$ from the sequence of natural numbers, we get:
$$a_{0}^{s_{k}}e_{0}+\ldots+a_{h-1}^{s_{k}}T^{h-1}e_{0}\underset{k\to+\infty}{\longrightarrow}a_{0}e_{0}+\ldots+a_{h-1}T^{h-1}e_{0}$$
and also
$$b_{0}^{s_{k}}e_{0}+\ldots+b_{h-1}^{s_{k}}T^{h-1}e_{0}\underset{k\to+\infty}{\longrightarrow}b_0 e_{0}+\ldots+b_{h-1}T^{h-1}e_{0}$$ 
where $(a_0,\ldots,a_{h-1}),(b_0,\ldots,b_{h-1})\in\mathbb{K}^{h}\setminus\{0\}$.
\\It follows that:
$$
\begin{aligned}
\left\Vert \frac{P_{s_{k}}(T)e_{0}}{\Vert P_{s_{k}}(T)e_{0}\Vert}-(a_0 e_{0}+\ldots+a_{h-1}T^{h-1}e_{0})\right\Vert
\underset{k\to+\infty}{\longrightarrow}0\end{aligned}$$
and
$$\begin{aligned}\left\Vert \frac{Q_{s_{k}}(T)e_{0}}{\Vert Q_{s_{k}}(T)e_{0}\Vert}-(b_0 e_{0}+\ldots+b_{h-1}T^{h-1}e_{0})\right\Vert\underset{k\to+\infty}{\longrightarrow}0.
\end{aligned}$$
Since the vectors $a_{0}e_{0}+\ldots+a_{h-1}T^{h-1}e_{0}$ and $b_0 e_{0}+\ldots+b_{h-1}T^{h-1}e_{0}$ have norm one, we can choose two integers $0\leq i,j\leq h-1$ such that $a_j\neq0$ and $b_i\neq 0$.
Then, we define two linear forms which are continuous for the product on $\K[T]e_0$: 
$$\Psi_{i,j}=\Phi_i+\Phi_{h+j}\text{ and }\widetilde{\Psi}_{i,j}=\Phi_i+2\Phi_{h+j}.$$
Thus, thanks to the continuity for the product, we have 
\begin{align}
\lim_{k\to\infty}\Psi_{i,j}\left(T^{m_{s_{k}}}e_{0}.P_{s_{k}}(T)e_{0}.\frac{Q_{s_{k}}(T)e_{0}}{\Vert Q_{s_k}(T)e_0\Vert}\right)&=\lim_{k\to\infty}\Psi_{i,j}\left(T^{m_{s_{k}}}P_{s_{k}}(T)e_{0}.\frac{Q_{s_{k}}(T)e_{0}}{\Vert Q_{s_k}(T)e_0\Vert}\right)&\label{equapsi1}\\
&=\Psi_{i,j}\left(e_{0}.(b_0e_0+\ldots+b_{h-1}T^{h-1}e_0)\right)&\nonumber\\
&=b_i\text{ by definition of }\Psi_{i,j}&\nonumber\\\nonumber
\end{align}
and also
\begin{align}
\lim_{k\to\infty}\Psi_{i,j}\left(T^{m_{s_{k}}}e_{0}.\frac{P_{s_{k}}(T)e_{0}}{\Vert P_{s_{k}}(T)e_{0}\Vert}.Q_{s_{k}}(T)e_{0}\right)&= \lim_{k\to\infty}\Psi_{i,j}\left(T^{m_{s_{k}}}Q_{s_{k}}(T)e_{0}.\frac{P_{s_{k}}(T)e_{0}}{\Vert P_{s_k}(T)e_0\Vert}\right)&\label{equapsi2}\\ &=\Psi_{i,j}\left(T^he_{0}.(a_0e_0+\ldots+a_{h-1}T^{h-1}e_0)\right) &\nonumber\\
&=\Psi_{i,j}\left(a_0T^he_{0}+\ldots+a_{h-1}T^{2h-1}e_{0}\right) &\nonumber\\
&=a_j.&\nonumber\\\nonumber
\end{align}
Doing the same thing with the second linear form $\widetilde{\Psi}_{i,j}$, we obtain similar results
\begin{equation}\label{equapsi3}
\lim_{k\to\infty}\widetilde{\Psi}_{i,j}\left(T^{m_{s_{k}}}e_{0}.P_{s_{k}}(T)e_{0}.\frac{Q_{s_{k}}(T)e_{0}}{\Vert Q_{s_{k}}(T)e_{0}\Vert}\right)=b_i
\end{equation}
and
\begin{equation}\label{equapsi4}
\lim_{k\to\infty}\widetilde{\Psi}_{i,j}\left(T^{m_{s_{k}}}e_{0}.\frac{P_{s_{k}}(T)e_{0}}{\Vert P_{s_{k}}(T)e_{0}\Vert}.Q_{s_{k}}(T)e_{0}\right)=2a_j.
\end{equation}
We now consider the quotient of equations (\ref{equapsi1}) and (\ref{equapsi2}) and of equations (\ref{equapsi3}) and (\ref{equapsi4}). This gives
$$
\xymatrix{
&&\frac{\Vert Q_{s_{k}}(T)e_{0}\Vert}{\Vert P_{s_{k}}(T)e_{0}\Vert} \ar@{=}[rd] \ar@{=}[ld]&\\
&\frac{\Psi_{i,j}\left(T^{m_{s_{k}}}e_{0}.\frac{P_{s_{k}}(T)e_{0}}{\Vert P_{s_{k}}(T)e_{0}\Vert}.Q_{s_{k}}(T)e_{0}\right)}{\Psi_{i,j}\left(T^{m_{s_{k}}}e_{0}.P_{s_{k}}(T)e_{0}.\frac{Q_{s_{k}}(T)e_{0}}{\Vert Q_{s_{k}}(T)e_{0}\Vert}\right)} \ar[d]^{k\to+\infty}&&\frac{\widetilde{\Psi}_{i,j}\left(T^{m_{s_{k}}}e_{0}.\frac{P_{s_{k}}(T)e_{0}}{\Vert P_{s_{k}}(T)e_{0}\Vert}.Q_{s_{k}}(T)e_{0}\right)}{\widetilde{\Psi}_{i,j}\left(T^{m_{s_{k}}}e_{0}.P_{s_{k}}(T)e_{0}.\frac{Q_{s_{k}}(T)e_{0}}{\Vert Q_{s_{k}}(T)e_{0}\Vert}\right)} \ar[d]^{k\to+\infty} \\
 &\frac{a_j}{b_i} \ar@{=}[rr]&&\frac{2a_j}{b_i} \\
}
$$
This equality contradicts the fact that $a_j\neq0$ and $b_i\neq0$.
This contradiction proves the lemma!
\end{proof}

Assume now and for the following that $X$ is a Banach space having a normalised unconditional basis $(e_{i})_{i\in\Z_{+}}$ for which the associated forward shift is continuous.
\\We set:
$$c_{00}=\Span\{e_{i},i\in\Z_{+}\}.$$\\
Since Lemma \ref{lemcritnfps} gives a criterion for checking non-strong $h$-supercyclicity, the proof of Theorem \ref{theosupnonpsup} reduces to the proof of the following points:

\begin{subequations}
\begin{gather}
\Span\{T^{i}e_{0},i\in\Z_{+}\}=\Span\{e_{i},i\in\Z_{+}\}\label{P1}.\\ 
\K[T]e_{0}\subseteq\overline{\{\lambda T^{i}e_{0},i\in\Z_{+},\lambda\in\K\}}\label{P2}.\\ 
T \text{ is continuous}.\label{P3}\\ 
\text{ There exist }2h\text{ linear forms }\Phi_\delta:\K[T](e_0)\to \K\text{ such that for every }\delta=0,\ldots,2h-1,\label{P4}\\
\left( P(T)e_{0}, Q(T)e_{0}\right)\to\Phi_\delta\left((PQ)(T)e_{0}\right)\text{ are continuous on }\K[T](e_0)\times \K[T](e_0),\nonumber\\
\text{For every }2\leq h\leq p\text{, every } \delta\in\{0,\ldots,2h-1\}\text{ and every }0\leq i\neq\delta\leq2h-1,\label{P5}\\
\Phi_\delta\left(T^{\delta}e_{0}\right)=1\text{ and }\Phi_\delta\left(T^{i}e_{0}\right)=0.\nonumber
\end{gather}
\end{subequations}

\subsubsection{Construction of $T$}

Our construction of $T$ is a modification of the construction of Bayart and Matheron \cite[section 4.2]{Bay}. We will give the definition of the operator and theorems leading to the continuity of $T$ but we will omit some proofs because they are the same as in \cite{Bay} up to some details.

Let us begin with a few terminology.
\\Set a countable dense subset $\mathbf{Q}$ of $\K$. A sequence of polynomials $\mathbf{P}=(P_{n})_{n\in\Z_{+}}$ is said to be \textit{admissible} if $P_{0}=0$ and $\mathbf{P}$ contains all polynomials whose coefficients are in $\mathbf{Q}$.
Let also $\deg(P)$ denote the degree of $P$, $\vert P\vert_{1}$ the sum of the moduli of its coefficients and $cd(P)$ its leading coefficient.
 We are going to construct $T$ as an almost weighted forward shift in order to satisfy (\ref{P1}). Actually, we need two sequences to construct $T$: the first one is the sequence of weights $(w_{n})_{n\in\Z_{+}}$ and the second one is a strictly increasing sequence $(b_{n})_{n\in\Z_{+}}$ indexing the iterates of $e_0$ for which the shift will be perturbed.\\
We define $T$ such that the iterates of $e_{0}$ corresponding to the perturbation satisfy (\ref{P2}):
\begin{equation}\label{pert}
 \text{for every }n\in\N,\ T^{b_{n}}e_{0}=P_{n}(T)e_{0}+e_{b_{n}}
\end{equation}
and we also define:
\begin{equation}\label{opnorm}
 \text{ for every }i\in[b_{n-1},b_{n}-1[\text{ and }n\in\N,\ T(e_{i})=w_{i+1}e_{i+1} 
\end{equation}
Thus we can express the vectors $Te_{b_{n}-1}$:
$$\begin{aligned}
 T^{b_{n}}e_{0}&=T^{b_{n}-b_{n-1}}T^{b_{n-1}}e_{0}&\\
&=T^{b_{n}-b_{n-1}}(P_{n-1}(T)e_{0}+e_{b_{n-1}})&\\
&=T^{b_{n}-b_{n-1}}P_{n-1}(T)e_{0}+w_{b_{n-1}+1}\ldots w_{b_{n}-1}Te_{b_{n}-1}& 
\end{aligned}$$
And replacing $T^{b_{n}}e_{0}$ with $P_{n}(T)e_{0}+e_{b_{n}}$ yields to:
$$Te_{b_{n}-1}=\varepsilon_{n}e_{b_{n}}+f_{n}$$
where
\begin{equation}\label{fn}
\varepsilon_{n}=\frac{1}{w_{b_{n-1}+1}\cdots w_{b_{n}-1}}\text{ and }f_{n}=\frac{1}{w_{b_{n-1}+1}\cdots w_{b_{n}-1}}(P_{n}(T)e_{0}-T^{b_{n}-b_{n-1}}P_{n-1}(T)e_{0}).
\end{equation}
Obviously, this definition is non-ambiguous if $\deg(P_{n})<b_{n}-1$. We assume now that $\deg(P_{n})<b_{n}-1$ for any $n\in\N$.

We also make the following choices for the values of $(b_{n})_{n\in\Z_{+}}$ and $(w_{n})_{n\in\Z_{+}}$:
$$
b_{0}=1\\,b_{n}=(2p+1)^{n}\text{ for every }n\in\N,\\
w_{n}=4\left(1-\frac{1}{2\sqrt{n}}\right)\text{ for every }n\in\N.\\
$$
The choice of $b_{n}$ is motivated by the fact that we will need $b_1>2p$ to check (\ref{P5}).
We denote also $d_{n}=\deg(P_{n})$.\\
Defined in this way, one can quickly check that $T:c_{00}\to c_{00}$ satisfies (\ref{P1}) and (\ref{P2}) by definition.

\subsubsection{Continuity of $T$}
We are now checking (\ref{P3}).
We introduce the following terminology: one will say that $\mathbf{P}$ is controlled by a sequence of natural numbers $(c_{n})_{n\in\Z_{+}}$ if for all $n\in\Z_{+}$, $\deg(P_{n})<c_{n}$ and $\vert P_{n}\vert_{1}\leq c_{n}$. This is an easy fact that if $\underset{n\to\infty}{\limsup}\ c_{n}=+\infty$ then there exists an admissible sequence which is controlled by $(c_{n})_{n\in\Z_{+}}$. 
We also denote $\Vert\sum_{i\in\Z_{+}} x_{i}e_{i}\Vert_{1}=\sum_{i\in\Z_{+}}\vert x_{i}\vert$ the $\ell^{1}$ norm on $c_{00}$.

To check (\ref{P3}), we need the following lemma which is almost the same as Lemma 4.20 from \cite[p.~87]{Bay}. The reader should refer to it for a proof:

\begin{lem}\label{lemmajor}
 The following properties are satisfied:
\begin{subequations}
 \begin{gather}
  \varepsilon_{n}\leq1 \text{ for all }n\in\N\label{sup1}\\
\text{If }n\in\N\text{ and if }\Vert f_{k}\Vert_{1}\leq1\text{ for every }k<n\text{ then }:\label{sup2}\\
\Vert f_{n}\Vert_{1}\leq4^{\max(d_{n},d_{n-1})+1}\left(\frac{\vert P_{n}\vert_{1}}{2^{b_{n-1}}}+\vert P_{n-1}\vert_{1}\exp\left(-c\sqrt{b_{n-1}}\right)\right)\text{ where }c>0\text{ is a numerical constant}.\nonumber
 \end{gather}
\end{subequations}
\end{lem}  

The following lemma proves that (\ref{P3}) is satisfied for an appropriate choice of $\mathbf{P}$. For the same reasons as before, see \cite[p.~88]{Bay} for a proof.

\begin{lem}\label{lemTcont}
 There exists a control sequence $(u_{n})_{n\in\Z_{+}}$ tending to infinity such that the following holds: if the sequence $\mathbf{P}$ is controlled by $(u_{n})_{n\in\Z_{+}}$ then $T$ is continuous on $c_{00}$ with respect to the topology of $X$.
\end{lem}

\subsubsection{Construction of $\Phi_\delta$}
Since we have completed the construction of $T$ and proved that it is continuous, we need to focus on the functionals $\Phi_\delta$ satisfying (\ref{P4}) and (\ref{P5}).
We have to define $2p$ maps $\Phi_{\delta}$, $\delta\in\{0,\ldots,2p-1\}$ continuous for the product on $\K[T]e_0$.
The following lemma from \cite[p.~89]{Bay} permits to check the continuity of such functionals more easily:

\begin{lem}\label{lemphicont}
Let $\phi$ be a linear functional on $c_{00}$. Suppose that $\sum_{r,q}\vert\phi(e_{r}\cdot e_{q})\vert<\infty$.Then, the map $(x,y)\mapsto\phi(x\cdot y)$ is continuous on $c_{00}\times c_{00}$.
\end{lem}

Let us construct the functionals $\Phi_{\delta}$.
In the following, a vector $x\in c_{00}$ is said to be supported on some set $I\subset \Z_{+}$ if $x\in \Span\{T^{i}e_{0},i\in I\}$.
We want to construct $\Phi_{\delta}$ satisfying Lemma \ref{lemphicont}. Thus, we have to be able to give an upper bound of $\vert\phi(e_{r}\cdot e_{q})\vert$.\\
Take $r\leq q$, and write $r=b_{k}+u$ and $q=b_{l}+v$ with $u\in\{0,\ldots,b_{k+1}-b_{k}-1\}$ and $v\in\{0,\ldots,b_{l+1}-b_{l}-1\}$.
By definition of $T$, we can re-express:
$$\begin{aligned}
e_{r}&=\frac{1}{w_{b_{k}+1}\cdots w_{b_{k}+u}}\left(T^{b_{k}}-P_{k}(T)\right)T^{u}(e_{0})&\\
e_{q}&=\frac{1}{w_{b_{l}+1}\cdots w_{b_{l}+u}}\left(T^{b_{l}}-P_{l}(T)\right)T^{v}(e_{0}).&\\
\end{aligned}$$
Hence, for any linear functional $\phi:c_{00}\to\K$, we have:
$$\vert\phi(e_{r}\cdot e_{q})\vert\leq\frac{1}{2^{u+v}}\vert\phi(y_{(k,u),(l,v)})\vert,$$
where $y_{(k,u),(l,v)}=(T^{b_{k}}-P_{k}(T))(T^{b_{l}}-P_{l}(T))T^{u+v}e_{0}$ because $w_{i}\geq2$ for every $i\in\Z_{+}$.\\
To ensure the convergence of the summation from Lemma \ref{lemphicont}, we set:
$$\Phi_{\delta}(T^{i}e_{0})=\begin{cases}
1\text{ if }i=\delta\\
0\text{ if }i\in\{0,b_{1}-1\}\setminus\{\delta\}\\
\Phi_{\delta}(P_{n}(T)T^{i-b_{n}}e_{0})\text{ if }i\in[b_{n},\frac{3}{2}b_{n}[\cup[2b_{n},\frac{5}{2}b_{n}[\\
0 \text{ otherwise.}
\end{cases}$$
Moreover, $\Phi_{\delta}$ is well defined on $c_{00}$ because $\deg(P_{n})+i-b_{n}<i$, hence $P_{n}(T)T^{i-b_{n}}e_{0}$ is supported in $\{0,\cdots,i-1\}$.

To ensure the continuity of $\Phi_{\delta}$, we need the following lemma which can be essentially found in \cite[p.~90]{Bay}.
\begin{lem}\label{lemphicont2}
Assume that $\deg(P_{n})<\frac{b_{n}}{3}$ for all $n\in\Z_{+}$. Then, the following properties hold whenever $0\leq k\leq l$.
\begin{subequations}
 \begin{gather}
\Phi_{\delta}(y_{(k,u),(l,v)})=0\text{ if }u+v<\frac{b_{l}}{6}\label{phicont2a}\\
\vert \Phi_{\delta}(y_{(k,u),(l,v)})\vert\leq M_{l}(\mathbf{P}):=\max_{0\leq j\leq l}(1+\vert P_{j}\vert_{1})^{2}\prod_{j=1}^{l+1}\max(1,\vert P_{j}\vert_{1})^{2}.\label{phicont2b}
 \end{gather}
\end{subequations}
\end{lem}

The next proposition \cite[p.~91]{Bay} makes use of the two previous lemmas to ensure the continuity of $(x,y)\mapsto \Phi_{\delta}(x\cdot y)$ if $\mathbf{P}$ is suitably chosen:
\begin{prop}\label{propphicont}
There exists a control sequence $(v_{n})_{n\in\Z_{+}}$ such that the following holds: if the enumeration $\mathbf{P}$ is controlled by $(v_{n})_{n\in\Z_{+}}$, then the map $(x,y)\mapsto \Phi_{\delta}(x\cdot y)$ is continuous on $c_{00}\times c_{00}$.
\end{prop}

Thus, by Lemmas \ref{lemTcont} and \ref{lemphicont2}, with a well-chosen control sequence, $T$ is continuous and $\Phi_{\delta}$, $\delta\in\{0,\ldots,2p-1\}$ are continuous for the product defined on $\K[T]e_{0}$ and therefore (\ref{P4}) is satisfied.
Then, a simple computation proves that $\Psi$ satisfies (\ref{P5}) because $\Phi_{\delta}(T^i e_0)=\begin{cases}
                                                                                                    1\text{ if }i=\delta\\
												    0\text{ if not}
                                                                                                   \end{cases}$ for every $0\leq i\leq b_1-1=2p$.

The combination of Lemma \ref{lemcritnfps}, Lemma \ref{lemTcont} and Proposition \ref{propphicont} completes the proof of Theorem \ref{theosupnonpsup}.

\subsection{A supercyclic operator which is not strongly $p$-supercyclic for any $p\geq2$}

The previous example of a supercyclic operator which is not strongly $h$-supercyclic for a fixed $h$ is answering the question of the existence of strongly $h$-supercyclic operators which are not strongly $(h+1)$-supercyclic in the particular case $h=1$. In fact, we can improve this result:
\begin{theo}\label{supnonnsup}
There exists a supercyclic operator which is not strongly $h$-supercyclic for any $h\geq2$.
\end{theo}
To achieve this construction, we take a direct sum of some previously constructed operators to ensure the non-strong $h$-supercyclicity of this operator.
The first part of this section is devoted to the construction of infinitely many operators $T_{p}$ which are supercyclic but not strongly $h$-supercyclic for $2\leq h\leq p$. Moreover, we want that these operators satisfy some more properties to be able to ``match`` them later. For that purpose we have to modify the parameter $(b_n)_{n\in\Z_{+}}$ and the admissible sequence of polynomials $\mathbf{P}$ that we used in the previous section. From now on, we take $b_n=5^n$ for every $n\in\N$ and $b_0=0$.

\subsubsection{Construction of many operators with different parameters}
Given infinitely many increasing control sequences $(u_{n}^{i})_{n\in\Z_{+}}$ satisfying $\lim_{n\to+\infty}\ u_{n}^{i}=+\infty$ for every $i\geq2$, then it is possible to consider an enumeration $(S_{n}^{i})_{n\in\Z_{+}}$ not necessarily bijective of $(\mathbf{Q}[X])^{i+1}\times \{0\}^{\Z_{+}}$ for every $i\in\Z_{+}$ with the following properties: for every $i,k,n\in\Z_{+}$, and $0\leq j\leq b_{i+1}$, $S_{j}^{i}(k)=0$, $\deg(S_{n}^{i}(k))\leq u_{n}^{k+2}$ and $\vert S_{n}^{i}(k)\vert_{1}\leq u_{n}^{k+2}$.
These enumerations will be useful to construct infinitely many admissible sequences $(\mathbf{P}^k)_{k\in\Z_{+}}$, providing a construction of the desired operator, with the procedure explained below.
\\For every $j\in\Z_{+}$ and every $n\in\left[\frac{j(j+1)}{2},\frac{(j+1)(j+2)}{2}\right[$, we define: $$Q_{n}=S_{\frac{j(j+3)}{2}-n}^{n-\frac{j(j+1)}{2}}\text{ and }Q_{0}=S_{0}^{0}.$$
We set also for every $k\geq2$ and every $n\in\Z_{+}$, $\mathbf{P}_{n}^{k}:=Q_{n}(k-2)$.
\\\includegraphics[trim = 0mm 35mm 00mm 13mm, clip,scale=0.75]{suiteadmissible1.pdf}
The construction of $(Q_n)_{n\in\Z_{+}}$ is made with two purposes in mind: on the one hand every element from $\mathbf{Q}[X]$ must appear once as the $k$-th component of some $Q_n$ where the other components are all zero and this has to be satisfied for any $k$. This property allows $T$ to be cyclic. On the other hand, to turn cyclicity into supercyclicity, for every element $P$ from $\mathbf{Q}[X]$, we need to be able to find infinitely many $Q_n$ containing repetitions of $\lambda P$ on their firsts components and zeros elsewhere where $\lambda$ and the number of repetitions grow with $n$.
\\Coming back to the previously defined sequences $\mathbf{P}^{k}$, we state that such sequences are admissible and controlled by $(u_{n}^{k})_{n\in\Z_{+}}$ for every $k\geq2$.
Indeed, for every $n\in\Z_{+}$ and $k\geq2$, $\mathbf{P}^{k}_{n}=Q_{n}(k-2)$ hence there exists $q,r\leq n$ such that $\mathbf{P}^{k}_{n}=S^{q}_{r}(k-2)$.
Then, $\deg(S^{q}_{r}(k-2))\leq u_{r}^{k}$ by definition and therefore $\deg(S^{q}_{r}(k-2))\leq u_{n}^{k}$ because $(u_{n}^{k})_{n\in\Z_{+}}$ is increasing. The same argument shows that $\vert\mathbf{P}_{n}^{k}\vert_{1}\leq u_{n}^{k}$.
\\For the admissibility, we first notice that $S_{0}^{0}(i)=0$ for all $i\in\Z_{+}$. Moreover for every $k\geq2$, $\mathbf{P}^k$ is an enumeration of $\mathbf{Q}[X]$ because $\mathbf{Q}[X]=\{S_{n}^{k-2}(k-2),n\in\Z_{+}\}\subseteq\{\mathbf{P}_{n}^{k},n\in\Z_{+}\}$. 

\claim
These sequences have the additional property that $\mathbf{P}_{j}^{k}=0$ for every $k\geq2$ and $0\leq j\leq 2k$.
\begin{proof}
Let $k\geq2$ and $0\leq j\leq 2k$, $\mathbf{P}_{j}^{k}=Q_{j}(k-2)=S_{r}^{q}(k-2)$ for some $q,r\in\Z_{+}$ with $q+r\leq j$ by definition of $Q$.
By definition, if $0\leq r\leq b_{q+1}$, $S_{r}^{q}(k-2)=0$.
\\Then, if $b_{q+1}<r$, we have $b_{q+1}+q<r+q\leq j\leq 2k$, giving $\frac{5^{q+1}+q-4}{2}\leq k-2$.
In addition, if $q\geq1$, an easy computation yields to $q+1<\frac{5^{q+1}+q-4}{2}$. Hence, we get $q+1<k-2$ and thus $S_{r}^{q}(k-2)=0$ because $S_{r}^{q}\in\left(\mathbf{Q}[X]\right)^{q+1}\times\{0\}^{\Z_{+}}$.
\\It remains to study the case with $q=0$ and $5=b_{1}<r$ but $S_{r}^{q}(k-2)=0$ if $k-2>1\Leftrightarrow k\neq2$.
Furthermore, if $k=2$, then $q\neq0$ because otherwise we would have the inequality $5=b_1<r\leq j\leq 2k=4$.
\\This finishes the proof of the claim.
\end{proof}

Assume that $X$ is a Banach space with an unconditional normalised basis $(e_{i})_{i\in\Z_{+}}$ for which the associated forward shift is continuous.
For every $p\geq2$, set $X_p:=X$, $(e_{i}^{p})_{i\in\Z_{+}}:=(e_{i})_{i\in\Z_{+}}$ the unconditional basis of $X_p$ and define an operator $T_p$ on $X_p$ in the same way we did it in the last section but with parameters $p$, $(b_n)_{n\in\Z_{+}}$ and with admissible sequence $\mathbf{P}^p$ constructed above.
The changes on some parameters do not interfere with conditions (\ref{P1}), (\ref{P2}) and (\ref{P3}) which are still satisfied, thus $T_p$ is well-defined and continuous on $X_p$. In addition, it appears from the proof of Lemma \ref{lemTcont} that $\Vert T_p\Vert\leq \sup(4C_{u}^{p}\Vert F_p\Vert,2C_{u}^{p})$ where $C_{u}^{p}$ is the unconditional constant of $(e_{i}^{p})_{i\in\Z_{+}}$ and $F_p$ is the forward shift on $X_p$.
\\The delicate part is the construction of the linear forms because we use the condition $b_1>2p$ to construct the functionals $\Phi_{\delta}$ and check (\ref{P5}).
Here we have chosen to take $b_n=5^n$ for every $n\in\N$, then we have changed the admissible sequence $\mathbf{P}^p$ to be able to construct the functionals $\Phi_{\delta}$. Indeed, the first components of $\mathbf{P}^p$ contains only zeros to compensate for the fact that $b_{1}<2p$.
We define for every $\delta\in\{0,\ldots,2p-1\}$:
$$\Phi_{\delta}(T^{i}e_{0})=\begin{cases}
1\text{ if }i=\delta\\
0\text{ if }i\in\{0,b_m-1\}\setminus\{\delta\}\\
\Phi_{\delta}(P_{n}(T)T^{i-b_{n}}e_{0})\text{ if }i\in[b_{n},\frac{3}{2}b_{n}[\cup[2b_{n},\frac{5}{2}b_{n}[\text{ for }n\geq m\\
0 \text{ otherwise.}
\end{cases}$$
where $m\in\N$ is such that $b_{m-1}<2p< b_m$.
\\Then, $\Phi_{\delta}$ is well-defined thanks to the claim and (\ref{P5}) and (\ref{P4}) are also satisfied.
As a consequence, despite some changes on the parameters $T_p$ is supercyclic and not strongly $h$-supercyclic for $2\leq h\leq p$ on $X_p$ for a suitable choice of increasing control sequences.

\subsubsection{Construction of the supercyclic operator being not strongly $h$-supercyclic, $h\geq2$}
The next step is to consider the direct sum of previously constructed operators to get rid of the strong $h$-supercyclicity.
Hence, define $T=\oplus_{\ell_{2}}T_p$ an operator on $B=\oplus_{\ell_{2}}X_p$ for $p\geq 2$.
Then, considering on the first part the weighted forward shift part $R$ of $T$ and then the perturbation part $K$, we get:
$$\Vert T\Vert\leq\Vert R\Vert+\Vert K\Vert\leq 4\sup_{p\geq2}(C_{u}^{p}\Vert F_p\Vert)+2\sup_{p\geq2}(C_{u}^{p})<+\infty$$
because for every $p\geq2$, $X_p=X$.
So, $T$ is continuous on $B$ and is not strongly $p$-supercyclic for $p\geq2$.
Thus, it suffices to prove that $T$ is supercyclic with supercyclic vector $\oplus_{\ell^{2}}\frac{e_{0}^{p}}{p}$.
For this purpose we are going to prove that $T$ satisfies the two following conditions:
\begin{subequations}
\begin{gather}
\left\{(0,\cdots,0,e_{i}^{p},0,\cdots),i\in\Z_{+},p\geq2\right\}\subset\overline{\Span\left\{T^{i}\left(\oplus_{\ell_2} \frac{e_{0}^{p}}{p}\right),i\in\Z_{+}\right\}}\label{P21}.\\ 
\Span\left\{T^{i}\left(\oplus_{\ell_2} \frac{e_{0}^{p}}{p}\right),i\in\Z_{+}\right\}\subseteq\overline{\left\{\lambda T^{i}\left(\oplus_{\ell_2} \frac{e_{0}^{p}}{p}\right),i\in\Z_{+},\lambda\in\K\right\}}\label{P22}.
\end{gather}
\end{subequations}
The condition (\ref{P22}) is satisfied with our construction of admissible sequences. Indeed, it suffices to prove (\ref{P22}) for $\Span_{\mathbf{Q}}\left\{T^{i}\left(\oplus \frac{e_{0}^{p}}{p}\right),i\in\Z_{+}\right\}$.
Let then $P\in\mathbf{Q}[X]$, there exists by definition of $Q_k$ three strictly increasing sequences of integers $(n_k)_{k\in\Z_{+}}$,$(m_k)_{k\in\Z_{+}}$ and $(\lambda_k)_{k\in\Z_{+}}$ such that for every $k$, $Q_{n_k}=\left(\underbrace{\lambda_k P,\ldots,\lambda_k P}_{m_k\text{ times}},0,\cdots\right)$.\\
Hence, $T^{b_{n_{k}}}\left(\oplus_{\ell_2}\frac{e_{0}^{p}}{p}\right)=\left(\lambda_k P(T)\left(\frac{e_{0}^{2}}{2}\right)+\frac{e_{b_{n_{k}}}^{2}}{2},\ldots,\lambda_kP(T)\left(\frac{e_{0}^{m_k+1}}{m_k+1}\right)+\frac{e_{b_{n_{k}}}^{m_k+1}}{m_k+1},\frac{e_{b_{n_{k}}}^{m_k+2}}{m_k+2},\cdots\right)$. \\Thus, 
\begin{align*}
&\left\Vert \frac{1}{\lambda_k}T^{b_{n_{k}}}\left(\oplus_{\ell_2}\frac{e_{0}^{p}}{p}\right)-P(T)\left(\oplus_{\ell_2} \frac{e_{0}^{p}}{p}\right)\right\Vert_{\ell_2}\\
=&\left\Vert\left(\frac{e_{b_{n_{k}}}^{2}}{2\lambda_k},\ldots,\frac{e_{b_{n_{k}}}^{m_k+1}}{(m_k+1)\lambda_k},\frac{e_{b_{n_{k}}}^{m_k+2}}{(m_k+2)\lambda_k}-P(T)\left(\frac{e_{0}^{m_k+2}}{m_k+2}\right),\ldots\right)\right\Vert_{\ell_2}\underset{k\to+\infty}{\longrightarrow}0.
\end{align*}
This proves (\ref{P22}).

We now focus on (\ref{P21}).
Let $i\in\Z_{+}$ and $q\geq2$. The definition of $Q_k$ and the supercyclicity of $T_p$ implies that there exists a strictly increasing sequence of integers $(n_k)_{k\in\Z_{+}}$ such that $Q_{n_{k}}=\left(0,\ldots,0,P_k,0,\ldots\right)$ for all $k\in\Z_{+}$ where $P_k$ is a polynomial such that: $P_k (T)e_{0}^{q}=\lambda_k e_{i}^{q}+\varepsilon_k$ where $(\lambda_k)_{k\in\Z_{+}}$ is a strictly increasing sequence of positive real numbers tending to $+\infty$ and $\Vert \varepsilon_k\Vert\underset{k\to+\infty}{\longrightarrow}0$.
Thus, 
\begin{align*}
&\left\Vert \frac{q}{\lambda_k}T^{b_{n_{k}}}\left(\oplus_{\ell_2}\frac{e_{0}^{p}}{p}\right)-\left(0,\ldots,0,e_{i}^{q},0,\ldots\right)\right\Vert_{\ell_2}\\
=&\left\Vert\left(\frac{q e_{b_{n_{k}}}^{2}}{ 2\lambda_k},\ldots,\frac{q e_{b_{n_{k}}}^{q-1}}{(q-1)\lambda_k},\frac{e_{b_{n_{k}}}^{q}}{\lambda_k}+\frac{1}{\lambda_k}P_k(T)\left(e_{0}^{q}\right)-e_{i}^{q},\frac{qe_{b_{n_{k}}}^{q+1}}{(q+1)\lambda_k},\ldots\right)\right\Vert_{\ell_2}&\\
=&\left\Vert\left(\frac{q e_{b_{n_{k}}}^{2}}{ 2\lambda_k},\ldots,\frac{q e_{b_{n_{k}}}^{q-1}}{(q-1)\lambda_k},\frac{e_{b_{n_{k}}}^{q}}{\lambda_k}+\varepsilon_k,\frac{qe_{b_{n_{k}}}^{q+1}}{(q+1)\lambda_k},\ldots\right)\right\Vert_{\ell_2}\underset{k\to+\infty}{\longrightarrow}0.& 
\end{align*}

This proves (\ref{P21}). So $T$ is supercyclic on $X$ without being strongly $h$-supercyclic for any $h\geq2$ proving Theorem \ref{supnonnsup}.

\begin{quest}
 Are strongly $n$-supercyclic operators also strongly $(n+1)$-supercyclic for $n\geq2$?
\end{quest}

\begin{quest}
 Does $T$ automatically satisfies the supercyclicity criterion if it is strongly $n$-supercyclic for any $n\geq1$? 
\end{quest}

 \scshape

\vglue0.3cm
\hglue0.02\linewidth\begin{minipage}{0.9\linewidth}
Romuald Ernst\\
{Laboratoire de~Math\'ematiques (UMR 6620)} de l'Universit\'e Blaise Pascal\\
Complexe universitaire des C\'ezeaux, 63171 Aubi\`ere Cedex, France \\
E-mail : \parbox[t]{0.5\linewidth}{\textit{Romuald.Ernst@math.univ-bpclermont.fr}}
\end{minipage}
\end{document}